\documentclass[oneside]{amsart}
\usepackage[foot]{amsaddr}
\usepackage{cite,amssymb,mathtools,enumerate,thmtools,tikz}
\newtheorem{theorem}{Theorem}[section]
\newtheorem{lemma}[theorem]{Lemma}
\newtheorem{proposition}[theorem]{Proposition}
\newtheorem{corollary}[theorem]{Corollary}
\newtheorem{assumption}[theorem]{Assumption}
\theoremstyle{remark}
\newtheorem{remark}[theorem]{Remark}
\newcommand{\ZZ}{\mathbb{Z}}
\newcommand{\PP}{\mathbf{P}}
\DeclarePairedDelimiter\ceil{\lceil}{\rceil}
\DeclarePairedDelimiter\floor{\lfloor}{\rfloor}

\begin{document}

\author{David Croydon$^1$}
\address{$^1$Department of Statistics, University of Warwick}
\email{d.a.croydon@warwick.ac.uk}
\author{Stephen Muirhead$^2$}
\address{$^2$Department of Mathematics, University College London}
\email{s.muirhead@ucl.ac.uk}
\title[Functional limit theorems for the BTM with slowly varying traps]{Functional limit theorems for the Bouchaud trap model with slowly varying traps}
\begin{abstract}
We consider the Bouchaud trap model on the integers in the case that the trap distribution has a slowly varying tail at infinity. Our main result is a functional limit theorem for the model under the annealed law, analogous to the functional limit theorems previously established in the literature in the case of integrable or regularly varying trap distribution. Reflecting the fact that the clock process is dominated in the limit by the contribution from the deepest-visited trap, the limit process for the model is a spatially-subordinated Brownian motion whose associated clock process is an extremal process.
\end{abstract}
\thanks{Stephen Muirhead was supported by a Graduate Research Scholarship from University College London and the Leverhulme Research Grant RPG-2012-608 held by Nadia Sidorova.}
\date{\today}
\maketitle

\section{Introduction}
The Bouchaud trap model (BTM) is of general interest in the study of stochastic processes due to its utility as a natural toy model for many diverse trapping phenomena (see, for example, the lecture notes of \cite{BenArous06}). Although applications often call for traps whose depths are integrable or regularly varying (including in the setting in which the BTM was introduced, see \cite{Bouchaud92}), the case of slowly varying traps is receiving growing attention in the literature. Indeed, slowly varying traps have been shown to arise naturally in the study of certain random walks in random environments, such as biased random walks on critical Galton-Watson trees \cite{Croydon13}, as well in the study of spin-glass dynamics on subexponential time scales \cite{BenArous12, Bovier13}. With regards to the BTM with slowly varying traps in particular, recent work has established the nature of localisation \cite{Muirhead14} and ageing \cite{Gun13} that can occur, which turns out to be qualitatively different from the equivalent phenomena in the case of integrable or regularly varying traps. The present work continues this study, giving a functional limit theorem for the BTM on the integers with slowly varying traps. Again, the resulting behaviour of the model with slowly varying traps is qualitatively different from the equivalent limits in the case of integrable or regularly varying traps.

Let us start by defining the BTM on the integers. First, suppose $\tau = (\tau_x)_{x \in \ZZ}$ is a collection of independent and identically-distributed (i.i.d.)\ strictly-positive random variables whose common distribution has a slowly varying tail, by which we mean that the non-decreasing function
\[ L(u) := \frac{1}{\PP(\tau_0 > u)}  \]
satisfies the \textit{slow-variation} property
\begin{align}
\label{eq:slow}
\lim_{ u \to \infty} \frac{ L(uv) }{L(u)} = 1, \quad \text{for any } v > 0.
\end{align}
We will refer to $\tau$ as the \textit{trapping landscape}; let $\PP$ denote its law and $\mathbf{E}$ the corresponding expectation. The BTM in the trapping landscape $\tau$ is the continuous-time Markov chain $X = (X_t)_{t \ge 0}$ on $\ZZ$ with transition rates
\[ w_{x \to y} := \begin{cases}
\frac{1}{2\tau_x},
  & \text{if } |x-y|=1, \\
0, &  \text{otherwise.}
 \end{cases}  \]
As always for random processes in random media, we distinguish between the \emph{quenched} and \emph{annealed} law of the BTM. For each realisation of the trapping landscape $\tau$, write $P_x^\tau$ for the law of the Markov chain with the above transition rates started from $x$; this is the \emph{quenched law} of the BTM. The corresponding \emph{annealed law} is then obtained as the semi-direct product
\[\mathbb{P}_x\left(\cdot\right)=\int {P}^\tau_x\left(\cdot\right)d\PP.\]
Our main result, which will be presented in Section \ref{mr} below, is to establish a functional limit theorem for $X$ under the annealed law $\mathbb{P}_0$. We note that quenched scaling limits (i.e.\ $P_x^\tau$-distributional scaling limits for $\PP$-a.e.\ trapping landscape) are not available in this setting, since there is no homogenisation of the trapping landscape under the relevant scaling.

The equivalent functional limit theorem in the case where $\tau_0$ has (i) integrable or (ii) regularly varying tails has previously been established in \cite{BenArous13}, building on the single-time scaling limit proved in \cite{FIN02}. In the case of integrable tails, the scaling limit for $X$ is just the standard Brownian motion, with the law of large numbers acting to smooth out the effect of the traps in the limit. In the case of regularly varying tails of index $\alpha \in (0, 1)$, the scaling limit for $X$ is a spatially-subordinated Brownian motion known as the FIN diffusion, which was introduced in \cite{FIN02}. Our scaling limit can be seen as the natural analogue of the FIN diffusion in the limiting case $\alpha = 0$; this statement is made precise in Theorem~\ref{thm:alpha0} presented below.

We also study a natural `transparent' generalisation of the BTM, whereby the random walk $X$, at each visit to a site $x \in \ZZ$, has a certain chance of `avoiding' the trap located at $x$. Transparent trap models were first considered in \cite{BenArous13}, where it was observed that very simple transparency mechanisms can yield a variety of diverse scaling limits. In particular, for a parameter $\beta \ge 0$ we define the $\beta$-transparent BTM as the (non-Markovian) symmetric nearest-neighbour random walk $X^\beta = (X^\beta_t)_{t \ge 0}$ on $\ZZ$ whose holding time on the $i^\mathrm{th}$ visit to a site $x \in \ZZ$ is distributed as
\[ \tau^i_x := \begin{cases}
 \tau_x \xi_x^i, & \text{with probability }  \min \left\{ \frac{1}{\tau_x^\beta}, 1 \right\} ,\\
\xi_x^i, & \text{otherwise},
\end{cases} \]
where $\{\xi_x^i\}_{x \in \mathbb{Z}, i \in \mathbb{N}}$ is a collection of i.i.d.\ unit-mean exponential random variables. In other words, the chain $X^\beta$ evolves just as the BTM $X$ except that a trap at $x \in \ZZ$ is ignored, at each visit to $x$, with probability $\max\{1 - \tau_x^{-\beta},0\}$ independently of all other sources of randomness; setting $\beta = 0$ recovers the BTM. In \cite{BenArous13}, a partition of the parameter space was obtained for the above model\footnote{Actually, in  \cite{BenArous13} the model considered had $\xi_x^i\equiv 1$, although this change does not affect the scaling limits.} in the case of regularly varying traps, indicating the rich variety of scaling limits that may arise; see Figure~\ref{fig}. In particular, each of the standard Brownian motion, the FIN diffusion, as well as the fractional kinetics (FK) process can arise through this transparency mechanism; see \cite{BenArous13} or \cite{BenArous06} for the definition and basic properties of the FK process. We consider the $\beta$-transparent BTM with slowly varying traps, confirming that the partition of the parameter space in Figure~\ref{fig} remains valid on the $\alpha = 0$ boundary, a case that was not treated in \cite{BenArous13}. Incidentally, this raises the interesting question as to what other processes, apart from FK and the scaling limit of the BTM, may arise as the scaling limit of transparent trap models as the parameters $\alpha, \beta \to 0$ simultaneously; this question will be addressed in upcoming work.
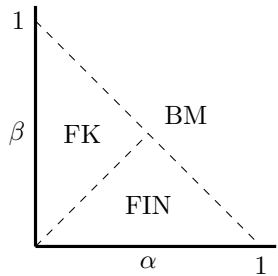
\begin{figure}[t]
\label{fig}
\begin{tikzpicture}
\draw [very thick] (0, 0) -- (1.5, 0) node[anchor=north] {$\alpha$} -- (3, 0)  node[anchor=north] {$1$} -- (3.2, 0);
\draw [very thick] (0, 0) -- (0, 1.5) node[anchor=east] {$\beta$} -- (0, 3)  node[anchor=east] {$1$} -- (0, 3.2) ;
\draw [thin, dashed] (0, 3) -- (3, 0);
\draw [thin, dashed] (0, 0) -- (1.5, 1.5);
\draw (1.5, 0.8) node[anchor=north] {\text{FIN}};
\draw (1, 1.5) node[anchor=east] {\text{FK}};
\draw (2, 2) node[anchor=north] {\text{BM}};
\end{tikzpicture}
\caption{Partition of the parameter space of the $\beta$-transparent trap model according to scaling limit due to \cite{BenArous13}. Here the labels `BM', `FIN' and `FK' indicate that the respective scaling limits are the standard Brownian motion, the FIN diffusion and the FK process. Recall that $\alpha$ denotes the index of regular-variation of the traps; slow-variation corresponds to the $\alpha = 0$ boundary.}
\end{figure}

\subsection{The scaling limit}
\label{subsec:scalinglimit}
In this section, we introduce the scaling limit of $X$ as a time-changed (or \textit{subordinated}) standard Brownian motion. Motivating this description, and key to the proof of our functional limit theorem, is the observation that every symmetric time-homogeneous nearest-neighbour continuous time random walk on $\ZZ$ can be expressed as a time-changed simple random walk, where the time-change may depend on the realisation of the underlying random walk (see, for example, \cite{BenArous06}). To see this for the process $X$, let $S = (S_i)_{i \in \mathbb{N}}$ be a discrete-time simple random walk (SRW) on $\ZZ$ and let $\xi = (\xi_i)_{i \in \mathbb{N}}$ be a collection of i.i.d.\ unit-mean exponential random variables, with $S$, $\xi$ and $\tau$ independent. Define an $S$-dependent \textit{clock process} $A = (A_n)_{n \ge 0}$ by setting
\[ A_n := \sum_{i \le \floor{n}} \xi_i \tau_{S_i}, \]
and let $I^S=(I^S_t)_{t\geq 0}$ be its right-continuous inverse, defined by
\[ I^S_t := \inf \{ n : A^S_n > t \}. \]
It is not hard to see that the law of $X$ under $\mathbb{P}_0$ is identical to that of $S_{I^S}$. In other words, the process $X$ may equivalently be defined via a subordination of the simple random walk $S$ by the clock process $A$.

We now present an analogous construction in the continuous setting that will serve to describe the scaling limit of $X$. Let $\mathcal{P} = (x_i, v_i)_{i \in \mathbb{N}}$ be an inhomogeneous Poisson point process on $\mathbb{R} \times \mathbb{R}^+$ with intensity measure $v^{-2} d x \, d v$; this process can be viewed as the scaling limit for the trapping landscape. Denote by $B = (B_t)_{t \ge 0}$ standard Brownian motion (independent of $\mathcal{P}$). Let $m^B = (m^B_t)_{t \ge 0}$ be the $B$-explored extremal process for $\mathcal{P}$, defined by
\begin{align}
m^B_t := \sup\left\{ v_i :  \inf_{s \in [0, t]} B_s \le x_i \le  \sup_{s \in [0, t]} B_s  \right\},\label{mtdef}
\end{align}
and let $I^B = (I^B_t)_{t \ge 0}$ be its right-continuous inverse, i.e.\ $I^B_t :=   \inf \{  s : m^B_s >  t \}$. We will identify the processes $m^B$ and $B_{I^B}$ as the scaling limits of $A$ and $X$ respectively.

We claim that $B_{I^B}$ is the natural analogue of the FIN diffusion with parameter $\alpha \in (0, 1)$ in the limiting case $\alpha = 0$. To make this precise, we first observe how the FIN diffusion can similarly be defined as the standard Brownian motion $B$, time-changed by a clock-process that is a function of $B$, the point-process $\mathcal{P}$ and the parameter $\alpha$. We then use this representation to show in Theorem~\ref{thm:alpha0} below that the FIN diffusion with parameter $\alpha$ converges almost-surely, under suitable rescaling, to the process $B_{I^B}$ as the parameter $\alpha \to 0$, where by almost-surely we mean with respect to the joint law of $\mathcal{P}$ and $B$. Because of this result, we refer to the process $B_{I^B}$ as the \textit{extremal FIN process}. (Note that we discuss the corresponding $\alpha\to1^-$ limit in Remark~\ref{alpha1rem}.)

Note that $B_{I^B}$ is a highly singular process. Indeed, conditional on $\mathcal{P}$, its probability mass is concentrated, at each time $t > 0$, on the two sites
\[ z_t^1 := \min \{ x_i \ge 0 : v_i > t \} \quad \text{and} \quad z_t^2 := \max \{ x_i \le 0 : v_i > t \},\]
in proportion to their hitting probability with respect to $B$. In other words, conditional on $\mathcal{P}$,
\[ B_{I^B_t} = \begin{cases}
 z_t^1 , &   \text{with probability } |z_t^1| / (z_t^1  -  z_t^2 ), \\
 z_t^2, & \text{with probability } |z_t^2| / (z_t^1  -  z_t^2 ).
 \end{cases} \]
That we have such localisation arising in the scaling limit is a consequence of the slowly varying tail of the trap distribution $\tau_0$, which means the clock process is dominated in the limit by the contribution from the deepest-visited trap so far. (As noted above, the localisation properties of the discrete model have previously been studied in \cite{Muirhead14}.)

\subsection{Main results}\label{mr}

We are now ready to state our functional limit theorem for the BTM, identifying $B_{I^B}$ as its scaling limit (see Theorem~\ref{thm:fltX} below). As in the cases of integrable or regularly varying traps, the key step towards this is establishing a functional limit theorem for the clock process $A$ (see Theorem~\ref{thm:fltA} below). We stress that both of these results hold for the BTM with arbitrary slowly varying tail distribution $\tau_0$.

Concerning topological issues, the conclusion of Theorem~\ref{thm:fltA} is stated in the Skorohod space of real-valued c\`{a}dl\`{a}g functions on $\mathbb{R}^+$, $D(\mathbb{R}^+)$, equipped with the Skorohod $M_1$ topology. For Theorem~\ref{thm:fltX}, we use the same state space, but the coarser non-Skorohod $L_{1,\rm{loc}}$ topology. We will later also refer to the usual Skorohod $J_1$ and uniform (over compact time intervals) topologies on $D(\mathbb{R}^+)$; detailed definitions of the $J_1$, $M_1$ and $L_{1, \rm{loc}}$ topologies and the relation between them are provided in Appendix \ref{sec:appendix}. As for notation, we write $\stackrel{J_1}{\Rightarrow}$, $\stackrel{M_1}{\Rightarrow}$ and $\stackrel{L_1}{\Rightarrow}$ for weak convergence in the $J_1$, $M_1$ and $L_{1,\rm{loc}}$ topologies respectively.

\begin{theorem}[Functional limit theorem for the clock process]
\label{thm:fltA} Under $\mathbb{P}_0$, as $n \to \infty$,
 \begin{equation}\label{alim}
 \left( \frac{1}{n} L \left( \frac{1}{n}  A_{n^2 t } \right) \right)_{t\geq 0} \stackrel{M_1}{\Rightarrow} \left(m^B_t  \right)_{t\geq 0}
 \end{equation}
in distribution.
\end{theorem}

\begin{theorem}[Functional limit theorem for the BTM]
\label{thm:fltX} Under $\mathbb{P}_0$, as $n \to \infty$,
\[ \left( \frac{1}{n} X_{n L^{-1}(n t)}\right)_{t\geq0} \stackrel{L_1}{\Rightarrow}\left( B_{I^B_t} \right)_{t\geq0}\]
in distribution (simultaneously with (\ref{alim})), where $L^{-1}$ is the right-continuous inverse of $L$, i.e.\ $L^{-1}(x) := \inf \{  u : L(u) > x \}$. 
\end{theorem}

\begin{remark}
\label{rem:topology}
Let us explain why the $M_1$ and $L_{1,\rm{loc}}$ topologies are the appropriate topologies for the convergence results in Theorems \ref{thm:fltA} and \ref{thm:fltX} respectively. Recall that the $M_1$ topology extends the usual $J_1$ topology by allowing jumps in the limit process to be matched by multiple jumps of lesser magnitude in the limiting processes, as long as they are essentially monotone and occur in negligible time in the limit. With regards to Theorem~\ref{thm:fltA}, the need for the $M_1$ topology arises because the total amount of time that the BTM spends at the deepest-visited trap is a result of multiple visits to the trap, all of which can contribute in a non-negligible way to the jump in the limit clock process. Convergence in the stronger $J_1$ topology would only hold if only the first visit to the trap made a non-negligible contribution in the limit; this is not true in general.

Recall also that the non-Skorohod $L_{1,\rm{loc}}$ topology extends both the $J_1$ and $M_1$ topologies by allowing excursions in the limiting processes that are not present in the limit process, as long as they are of negligible magnitude in the $L_1$ sense (which, in particular, is the case if they are of bounded size and occur in negligible time in the limit). In regards to Theorem~\ref{thm:fltX}, the need for the $L_{1,\rm{loc}}$ topology arises because, during the time that the BTM is based at the deepest-visited trap, the BTM makes repeated excursions away from this site. Although these occur in negligible time in the limit, they are of a magnitude comparable to the distance scale, and so prevent convergence in the stronger Skorohod topologies. See \cite{Fontes13,Mathieu14} for other examples of trap model convergence results that make use of the $L_{1,\rm{loc}}$ topology (or close variants). We remark that the convergence in the $L_{1,\rm{loc}}$ topology is too weak to imply the convergence of some commonly used functionals of the sample paths of $X$, including $\inf_{t \in [0, T]}(X_t)$ and $\sup_{t \in [0, T]}(X_t)$.

Finally, we believe that the convergence in the $L_{1, \rm{loc}}$ topology in Theorem~\ref{thm:fltX} can actually be mildly strengthened to convergence in a topology that allows for zero-time excursions of bounded size in the limiting processes, but only if they occur at jump-times of the limit process (cf.\ the space $E$ in \cite[Section 15.4]{Whitt02}). That we expect convergence to hold in this stronger topology is essentially due to the highly non-linear rescaling of time in the limit; since this is an artefact of the rescaling rather than an intrinsic property of the processes, we choose not to pursue the additional technical complications necessary to prove such a result here.
\end{remark}

We now state an additional assumption on the tail of the probability distribution $\tau_0$, under which we are able to derive the convergence of the clock process under the stronger $J_1$ topology, and also simplify the statements of Theorems \ref{thm:fltA} and \ref{thm:fltX}.

\begin{assumption}
\label{assumpt}
Assume that
\[ \lim_{x \to \infty} \frac{L(x / L(x))}{L(x)} = 1.\]
\end{assumption}

\begin{remark}
Note that Assumption \ref{assumpt} is satisfied for $L(x) = (\log x)^\gamma$ for all $\gamma > 0$ (i.e.\ the class of \textit{log-Pareto} distributions), but is only satisfied for $L(x) = \exp \{ (\log x)^\gamma \}$ (i.e.\ the class of \textit{log-Weibull} distributions) for the parameter range $0 < \gamma < 1/2$; cf.\ \cite[Remark 2.4]{Kasahara86}, where it is noted that the same parameter range allows for an analogous simplification of the functional limit theorem that we state as Proposition~\ref{prop:j1conv}. We additionally note that Assumption \ref{assumpt} implies the slow-variation property \eqref{eq:slow}.
\end{remark}

\begin{theorem}[Simplified functional limit theorem for the clock process]
\label{thm:assumptA}
Suppose Assumption \ref{assumpt} holds. Under $\mathbb{P}_0$, as $n \to \infty$,
\begin{equation}\label{alim2}
 \left( \frac{1}{n} L(A_{n^2 t})\right)_{t\geq0} \stackrel{J_1}{\Rightarrow} \left(m^B_t\right)_{t\geq 0}
\end{equation}
in distribution.
\end{theorem}

\begin{theorem}[Simplified functional limit theorem for the BTM]
\label{thm:assumptX} Suppose Assumption \ref{assumpt} holds. Under $\mathbb{P}_0$, as $n \to \infty$,
\[ \left(\frac{1}{n} X_{L^{-1}(n t)}\right)_{t\geq0}   \stackrel{L_1}{\Rightarrow}\left( B_{I^B_t}\right)_{t\geq0} \]
in distribution (simultaneously with (\ref{alim2})).
\end{theorem}

\begin{remark} That Assumption \ref{assumpt} ensures that the convergence of the clock process in Theorem~\ref{thm:assumptA} holds in the stronger $J_1$ topology results from the fact that, under this assumption, only the first visit to the deepest-visited trap makes a non-negligible contribution to the clock process in the scaling limit; c.f.\ Remark~\ref{rem:topology}.
\end{remark}

As described in Section \ref{subsec:scalinglimit}, the scaling limit process $B_{I^B}$ in Theorems \ref{thm:fltX} and \ref{thm:assumptX} can be considered as the extremal FIN process, in the sense that the FIN diffusion with parameter $\alpha$ converges, under suitable rescaling,
to the process $B_{I^B}$ as the parameter $\alpha \to 0$. Here we make this notion precise. First, let us define the FIN diffusion (see \cite{BenArous06}). Recall the definition of $\mathcal{P} = (x_i, v_i)_{i \in \mathbb{N}}$ and $B$ from Section \ref{subsec:scalinglimit}. For $\alpha \in (0, 1)$, let $\mathcal{P}^\alpha$ be the point process on $\mathbb{R} \times \mathbb{R}^+$ defined by the point set $(x_i,w_i):=(x_i, v_i^{1/\alpha})$, ${i \in \mathbb{N}}$. By simple change of variables, it is easy to see that $\mathcal{P}^\alpha$ is Poissonian with intensity measure $\alpha w^{-1 - \alpha} dx \, dw$. Denote by $(L_t(x))_{t\geq0,x\in\mathbb{R}}$ the local time process of $B$. Defining $m^{B, \alpha} = (m_t^{B, \alpha})_{t \ge 0}$ by
\begin{equation}\label{mtadef}
 m_t^{B, \alpha} := \sum_i L_t(x_i) v^{1/\alpha}_i
 \end{equation}
with $I^{B, \alpha} = (I^{B, \alpha}_t)_{t \ge 0}$ its right-continuous inverse, the FIN diffusion with parameter $\alpha$ is the process $B_{I^{B, \alpha}} = (B_{I^{B, \alpha}_t})_{t \ge 0} $. Note that in this definition we have built in a coupling, via the point process $\mathcal{P}$ and the Brownian motion $B$, of the FIN diffusion $B_{I^{B, \alpha}}$ and the extremal FIN process $B_{I^B}$. This allows us to state our convergence as an almost sure result (cf.\ the $\alpha\to1^-$ limit of Remark~\ref{alpha1rem}).

\begin{theorem}[Convergence of the FIN diffusion to the extremal FIN process]
\label{thm:alpha0}
As $\alpha \to 0$,
\[   \left(B_{ I^{B, \alpha}_{t^{1/\alpha}}} \right)_{t \ge 0} \stackrel{L_1 }{\rightarrow} \left( B_{I^B_t} \right)_{t\geq0},\]
where $\stackrel{L_1 }{\rightarrow}$ denotes convergence in the $L_{1, \rm{loc}}$ topology, almost-surely with respect to the joint law of $\mathcal{P}$ and $B$.
\end{theorem}

Finally, we establish scaling limits for the $\beta$-transparent BTM. Write $P_x^{\tau, \beta}$ for the quenched law of $X^\beta$ started from $x$, and let
\[\mathbb{P}^\beta_x\left(\cdot\right)=\int {P}^{\tau, \beta}_x\left(\cdot\right)d\PP\]
be the corresponding annealed law. To state our results, we will need to introduce the notion of second-order slow-variation. We say that $L$ is \textit{second-order slowly varying} if there exist functions $g, k$ such that $g(u) \to 0$ as $u \to \infty$ and
\begin{align}\label{sosv}
\lim_{u \to \infty} \frac{\frac{L(u v)}{L(u)} - 1}{ g(u)}  =  k(v), \quad \text{for any } v > 0,
\end{align}
where there exists a $v$ such that $k(v)\neq 0$ and $k(uv)\neq k(u)$ for all $u>0$. Second-order slow-variation is a natural strengthening of the slow-variation property \eqref{eq:slow}, giving more precise information about the fluctuations of $L$ at infinity; see \cite[Chapter 3]{Bingham87} for an overview.

\begin{theorem}[Functional limit theorem for the $\beta$-transparent BTM]
\label{thm:tt}
Suppose $\beta \ge 1$. Then, as $n \to \infty$,
\[ \left( \frac{1}{n} X^\beta_{\mu n^2 t} \right)_{t \ge 0} \stackrel{J_1}{\Rightarrow} (B_t)_{t \ge 0} \]
in $P_x^{\tau, \beta}$-distribution, $\PP$-almost-surely, where
\[ \mu := \mathbf{E} \left( \max\{1 - \tau_0^{- \beta}, 0\}  + \tau_0 \min\{\tau_0^{- \beta}, 1 \} \right) < \infty\]
is the annealed expected holding time. Suppose $\beta \in (0, 1)$ and assume that $L$ is second-order slowly varying, i.e.\ that it satisfies (\ref{sosv}). Then there exists a slowly varying function $\ell$ such that, as $n \to \infty$,
\[ \left( \frac{1}{n} X^\beta_{ n^{2/\beta} \ell(n) t} \right)_{t \ge 0} \stackrel{J_1}{\Rightarrow} (FK^\beta_t)_{t \ge 0} \]
in $P_x^{\tau, \beta}$-distribution and $\PP$-probability,
where $FK^\beta$ denotes the FK process with index $\beta$.
\end{theorem}

\begin{remark}
Note that the results in Theorem~\ref{thm:tt} agree with the partition of the parameter space depicted in Figure~\ref{fig}, where slow-variation is considered as the boundary case $\alpha = 0$.
\end{remark}

\begin{remark}
\label{rem:svfrep} The slowly varying function $\ell$ has an explicit representation in terms of $g$, the rate of second-order slow-variation of $L$, which in our setting is necessarily also slowly varying (see the proof of Proposition~\ref{prop:sosv}). More precisely, defining $ \Gamma(n) := n^{-\beta} g(n) / L(n)$, then $ \ell(n) = n^{-2/\beta} \Gamma^{-1}(n^{-2})$, where $\Gamma^{-1}$ denotes the right-continuous inverse of $\Gamma$. (NB.\ Without loss of generality, we may assume that $g$ is a strictly positive function and set $\Gamma^{-1}(t) :=\inf\{s:\Gamma(s) < t\}$.)
This representation explains why second-order slow-variation is a natural condition to impose in Theorem~\ref{thm:tt}.
\end{remark}

The rest of the paper is organised as follows. In Section \ref{sec:prelim} we collect preliminary results on general extremal processes and local times of simple random walks. Applying these results, in Section \ref{sec:clock} we start by studying the clock process $A$, establishing the functional limit theorems of Theorems \ref{thm:fltA} and \ref{thm:assumptA}. In the same section, we also consider the behaviour of the BTM itself, completing the proof of Theorems \ref{thm:fltX} and \ref{thm:assumptX}. In Section \ref{sec:alpha0} we prove the convergence of the FIN diffusion to the extremal FIN process of Theorem~\ref{thm:alpha0}. In Section \ref{sec:trans} we consider the $\beta$-transparent BTM, establishing Theorem~\ref{thm:tt}. Finally, in Appendix \ref{sec:appendix} we collect material related to the topologies $J_1$, $M_1$ and $L_{1, \rm{loc}}$, stating and proving some general convergence results that will be crucial for our proofs.

\section{Preliminary results}
\label{sec:prelim}

In this section we collect some preliminary results that will be used in proving the main theorems. When we describe a collection $(X_i)_{i \in I}$ of non-negative random variables as being \textit{bounded above} or \textit{bounded below} in probability we mean that $(X_i)_{i\in I}$ or $(1/X_i)_{i\in I}$, respectively, is tight.

\subsection{Extremal processes}
\label{subsec:extremal}

Let $E = (E_i)_{i \in \mathbb{N}}$ be a sequence of strictly positive i.i.d.\ random variables with arbitrary common distribution given by $\PP(E_i > u) := 1/L_E(u)$ for some non-decreasing, c\`{a}dl\`{a}g function $L_E$. Let $M = (M_n)_{n \ge 0}$ and $\Sigma = (\Sigma_n)_{n \ge 0}$ be respectively the extremal and sum processes for the sequence $E$, i.e.\
\[M_n := \max \{ E_i :  i \le \floor{n}  \} \quad \text{and} \quad  \Sigma_n = \sum_{i \le \floor{n}} E_i.\]
Further, let $\mathcal{J}$ denote the jump-set of the process $M$, i.e.\
\[ \mathcal{J} := \{ n : M_n \neq M_{n^-} \} \subseteq \mathbb{N}.\]

\begin{proposition}[Jump-set spacing]
\label{prop:jumps}
For each $C > 0$, as $n \to \infty$,
\[  |\mathcal{J} \cap (n / C, n C ] \, | \quad \text{and} \quad \frac { {\rm{sep}}( \mathcal{J} \cap (n / C, n C ] ) }{ n }   \]
are respectively bounded above and bounded below in probability, where ${\rm{sep}}(\mathcal{S})$ denotes the \textit{separation} of the set $\mathcal{S}$
\[ {\rm{sep}}(\mathcal{S}) : = \min_{\substack{ i, j \in \mathcal{S} \\ i \neq j} } |i-j|.\]
\end{proposition}
\begin{proof}
Let $\psi = \left({\psi}_i \right)_{i \in \mathbb{N}}$ be a sequence of i.i.d.\ unit-mean exponential random variables, with $M^{\psi}_n$ the extremal processes for $\psi$, i.e.\ $M^\psi_n := \max \{ \psi_i :  i \le \floor{n}  \}$, and $\mathcal{J}^\psi$ its associated jump-set, i.e.\ $ \mathcal{J}^\psi := \{ n : M^\psi_n \neq M^\psi_{n^-} \}$. Denote by $(k^n_i)_{i\geq 1}$ the ordered list of indices $i> n/C$ such that $\psi_i\geq M^\psi_{n/C}$ and abbreviate $K_n:=|\{i:k_i^n\leq nC\}|$. Clearly we have that
\[ |\mathcal{J^\psi} \cap (n / C, n C ] \, | \le K_n\]
and
\[ {\rm{sep}}( \mathcal{J^\psi} \cap (n / C, n C ] )  \geq {\rm{sep}}( k_i^n : i = 1, 2, \ldots, K_n ). \]
Moreover, by the inverse transform theorem,
\begin{equation}\label{itt}
E_n \stackrel{d}{=} L_E^{-1}(\exp \{ \psi_n \})
\end{equation}
and so there is a coupling of the sequences $E$ and $\psi$ such that, for all $n$,
\[ M_n \neq M_{n^-}  \implies  M^{\psi}_n \neq M^{\psi}_{n^-}. \]
Therefore, for this coupling, $ \mathcal{J} \subseteq \mathcal{J}^\psi$,
and so it is sufficient to prove that
\[  K_n \quad \text{and} \quad \frac{ {\rm{sep}}( k_i : i = 1, 2, \ldots, K_n ) }{ n }   \]
are respectively bounded above and bounded below in probability.

For the first, note that, conditionally on $M^\psi_{n/C}$, the random variable $K_n$ is distributed as
\begin{align}
\label{eq:jumps}
\text{Bi} \left( \floor{nC}, \exp\{ - M^\psi_{n/C} \} \right),
\end{align}
where $\text{Bi}(n, p)$ denotes a binomial random variable with parameters $n$ and $p$. It is a classical result of extreme-value theory (see, for example, \cite{Resnick87}) that, as $n \to \infty$,
\[M^\psi_{n} - \log n  \Rightarrow \mathcal{G} \quad \text{in distribution,} \]
where $\mathcal{G}$ denotes the Gumbel distribution, and so
\begin{align}
\label{eq:jumps2}
n \exp\{ - M^\psi_{n/C} \}
\end{align}
is bounded above in probability. Together with equation \eqref{eq:jumps} and Markov's inequality, this implies that $K_n$ is also bounded above in probability.

For the second, note that, conditionally on $M^\psi_{n/C}$, for any $i\geq 1$, the distance $k^n_{i+1} - k^n_i$ is distributed as $ \text{Geo} (\exp\{ - M^\psi_{n/C} \})$, where $\text{Geo}(p)$ denotes a geometric random variable (with support $1,2,\dots$). Again, by equation \eqref{eq:jumps2}, this implies that $n^{-1}(k^n_{i+1}- k^n_i)$ is bounded below in probability. By applying a union bound (conditional on $K_n$, which we already know is bounded above in probability), we get the result.
\end{proof}

Under the assumption that $L_E$ is slowly varying, it turns out that the processes $M$ and $\Sigma$ have scaling limits that coincide. Let $\mathcal{P}^+ := \mathcal{P}|_{\mathbb{R}^+ \times \mathbb{R}^+} $ denote the point process $\mathcal{P}$ defined in Section \ref{subsec:scalinglimit} restricted to the upper-right-quadrant, and let $m = (m_t)_{t \ge 0}$ be the extremal process for $\mathcal{P}^+$, that is
\[ m_t := \max \{ v_i : 0 \le x_i \le t\}.\]

\begin{proposition}[Functional limit theorem for the extremal and sum processes; \cite{Lamperti64, Kasahara86}]
\label{prop:j1conv} Assume $L_E$ satisfies the slow-variation property \eqref{eq:slow}. Then, as $n \to \infty$,
\[\left( \frac{1}{n} L_E (M_{n t } )\right)_{t\geq 0} \stackrel{J_1}{\Rightarrow}\left( m_t \right)_{t\geq 0} \quad \text{and} \quad \left(\frac{1}{n} L_E (\Sigma_{n t} ) \right)_{t\geq 0} \stackrel{J_1}{\Rightarrow} \left(m_t \right)_{t\geq 0}\]
in distribution.
\end{proposition}
\begin{proof}
The limit theorem for $\Sigma$ is the main result of \cite{Kasahara86}. The limit theorem for $M$ may be derived by first applying \cite[Theorems 2.1 and 3.2]{Lamperti64} to the random variables $(\psi_i)_{i\geq 1}$ introduced in the proof of Proposition~\ref{prop:jumps}, which yields
\begin{equation}\label{lim1}
M_{nt}^\psi-\log n\stackrel{J_1}{\Rightarrow} \log m_t
\end{equation}
(where we note that the limit process $(m(t))_{t\geq0}$ in \cite{Lamperti64} is $(\log m_t)_{t \ge 0}$ in our notation), and then transforming using (\ref{itt}). For the final step, one should be slightly careful since $L_E(E_i)$ is not identically distributed as $e^{\psi_i}$ in general. However, we do have that
\[\frac{1}{n} L_E \left(M_{n t }\right)\stackrel{d}{=}
\frac{L_E\left(L_E^{-1}\left(e^{M_{nt}^\psi}\right)\right)}{e^{M_{nt}^\psi}}\times \frac{{e^{M_{nt}^\psi}}}{n}\]
as processes. By (\ref{lim1}), the second product converges in distribution to $(m_t)_{t\geq0}$ in the $J_1$ topology. As for the effect of multiplying by the first term, this can then be controlled using the facts that $e^{M^\psi_n}\to\infty$ almost-surely, $L_E(L_E^{-1}(x))\sim x$ as $x\to \infty$ (since $L_E$ is slowly varying), and also, for any $\varepsilon>0$,
\[\lim_{t\to0}\lim_{n\to\infty}\mathbf{P}\left(n^{-1}L_E \left(M_{n t }\right)\geq\varepsilon \right)\leq \lim_{t\to0}\lim_{n\to\infty}\mathbf{P}\left(n^{-1}L_E \left(\Sigma_{n t }\right)\geq\varepsilon\right)=0\]
by the first part of the proposition.
\end{proof}

\subsection{Simple random walks and local time}
\label{subsec:srw}
For $R = (R_i)_{i\geq0}$ a discrete-time simple random walk on $\ZZ$, let
\[d_n := \max_{i \le \floor{n}} R_i - \min_{i \le \floor{n}} R_i  \]
be the associated diffusion distance. We then have the following as a simple consequence of Donsker's invariance principle.

\begin{proposition}[Bounds for the diffusion distance of a SRW]
\label{prop:diffdist}
As $n \to \infty$,
\[  \frac{  d_n }{\sqrt{n}} \]
is bounded below and above in probability.
\end{proposition}

Next, let $\nu(n, x)$ be the (continuous) local time of a continuous-time simple random walk (CTSRW) after $n$ steps,
\[ \nu(n, x) :=  \sum_{ \{0 \le i \le \floor{n}  : R_i = x \} } \psi_i, \]
where $\psi = (\psi_i)_{i \in \mathbb{N}}$ is a sequence of i.i.d.\ unit-mean exponential distributions, and let $\nu_{\text{max}}(n)$ and $\nu_{\text{min}}(n)$ be respectively the maximum and minimum local times among the sites visited by $R$ after $n$ steps
\[ \nu_{\text{max}}(n) := \max_{ x \in \{R_i : \, i \le n\} } \nu(n, x) := \max_{x \in \ZZ} \nu(n, x);\]
\[ \nu_{\text{min}}(n) := \min_{ x \in \{R_i : \, i \le n\} } \nu(n, x).\]
We complete this subsection by deriving the bounds we will need for these quantities.

\begin{proposition}[Bounds for the local time of a CTSRW]
\label{prop:ctslocaltime}
As $n \to \infty$,
\[ \frac{ \nu(n, 0) }{ \sqrt{n} }  \quad \text{and} \quad  \frac{\nu_{\text{max}}(n)}{\sqrt{n}}\]
are both bounded below and above in probability.
\end{proposition}
\begin{proof} If $(B_t)_{t\geq 0}$ is a Brownian motion, and $(L_t(x))_{t\geq0,x\in\mathbb{R}}$ is its local time process, then it is standard that $L_{\sigma_{\pm1}}(0)$ has an  exponential distribution with mean one, where $\sigma_{\pm1}$ is the first hitting time of $\pm1$. It follows that $(\nu(n,x))_{n\geq 0,x\in \mathbb{Z}}$ has the same distribution as $(L_{\sigma(n)}(x))_{n\geq 0,x\in\mathbb{Z}}$, where $\sigma(0)=0$ and, for $n\geq1$, $\sigma(n):=\inf\{t>\sigma(n-1):\:B_t\in \mathbb{Z}\backslash\{B_{\sigma(n-1)}\}\}$. Since $n^{-1}\sigma(n)\rightarrow 1$ almost-surely, it follows that
\[\frac{\nu(n,0)}{\sqrt{n}}\Rightarrow L_1(0),\qquad\frac{\nu_{\text{max}}(n)}{\sqrt{n}}\Rightarrow \sup_{x\in\mathbb{R}}L_1(x),\]
in distribution, where to deduce this it is also helpful to recall the scaling property of Brownian local times, i.e.\ $(L_t(x))_{t\geq0,x\in\mathbb{R}}\stackrel{d}{=}(\lambda^{-1/2}L_{\lambda t}(\lambda^{1/2}x))_{t\geq0,x\in\mathbb{R}}$. The result follows.
\end{proof}

\begin{proposition}[Bound for the minimum local time of a CTSRW]
\label{prop:ctslocaltimemin}
For any $T > \delta > 0$, as $n \to \infty$,
\[  n \min_{ i \in [n^2 \delta, n^2 T] } \nu_{\text{min}}(i) \]
is bounded below in probability.
\end{proposition}
\begin{proof}
Combine the identity
\[ n \min_{i \le n} \psi_i \stackrel{d}{=} \psi_0  \ , \quad n \in \mathbb{N} \] with the bounds on the diffusion distance of Proposition~\ref{prop:diffdist}.
\end{proof}

\section{Scaling limits for the clock process and the BTM}
\label{sec:clock}

In this section we prove the convergence of the clock process $A$ to the limit process $m^B$ that is stated above in Theorems \ref{thm:fltA} and \ref{thm:assumptA}. The strategy is to `squeeze' the clock process $A$ between the $S$-explored extremal and sum processes defined respectively by
\[ M^X_n := \max \left\{ \tau_x : \min_{i \le \floor{n}} S_n \le x \le \max_{i \le \floor{n}} S_n \right\}  \quad \text{and} \quad \Sigma^X_n := \sum_{x=\min_{i \le \floor{n}} S_n}^{\max_{i \le \floor{n}} S_n}\tau_x.\]
We then apply a general squeeze convergence result for the Skorohod $M_1$ topology that we state and prove in Appendix \ref{sec:appendix} to complete the proof.

Throughout this section, fix constants $T > \delta > 0$. For technical reasons, we will additionally define an auxiliary function $h_n \to \infty$ growing sufficiently slowly that
\begin{align}
\label{eq:decay1}
\lim_{n \to \infty} \frac{L( L^{-1}(n / h_n) / h_n)}{L(L^{-1}(n / h_n))}  = \lim_{n \to \infty} \frac{L( L^{-1}(n / h_n) \, h_n)}{L(L^{-1}(n/h_n)) }  =  1,
\end{align}
remarking that such an $h_n$ is guaranteed to exist by the slow-variation property \eqref{eq:slow} and since $\lim_{n \to \infty} L^{-1}(n) = \infty$. For completeness, we give an explicit construction of an $h_n$ satisfying the left-hand side of equation \eqref{eq:decay1}; the construction of an $h_n$ that simultaneously satisfies the right-hand side of equation \eqref{eq:decay1} is analogous. Define an arbitrary increasing sequence $c = (c_i)_{i \in \mathbb{N}} \to \infty$, and denote, for each $u > 0$ and each $n > 0$,
\[ f^n(u) := \frac{ L( L^{-1}(n u) u) }{L(L^{-1}(n u))}.\]
By the slow-variation property and since $\lim_{n \to \infty} L^{-1}(n) = \infty$, we have that $f^n(u) \to 1$ for each $u$. This means that, for each $i \in \mathbb{N}$, there exists an $n_i \in \mathbb{N}$ such that
\[ | 1 - f^n(1/c_i)| < 1/c_i \quad \text{for all } n \ge n_i.\]
So define $h_n$, with increments only on the set $\{n_i\}_{i \in \mathbb{N}}$, satisfying $h_{n_i} := c_i$.

Similarly, under Assumption \ref{assumpt}, we will additionally require that $h_n$ satisfies
\begin{align}
\label{eq:decay2}
\lim_{n \to \infty} \frac{L( L^{-1}(n / h_n) / (n  h_n) )  }{L(L^{-1}(n/ h_n))}  = \lim_{n \to \infty} \frac{L( L^{-1}(n / h_n)  \, n \,  h_n) }{L(L^{-1}(n/ h_n))}  =  1,
\end{align}
which is again guaranteed to exist under Assumption \ref{assumpt} by analogous reasoning.

\subsection{Extremal processes associated to the BTM}
The first step is to convert our results on general extremal processes stated in Section \ref{sec:prelim} into equivalent results for the extremal processes associated with the random walk $X$.

Let $\mathcal{J}^X$ be the jump-set associated to $M^X$
\[ \mathcal{J}^X := \{n : M^X_{n} \neq M^X_{n^-} \}  \subseteq \mathbb{N}.\]
Abbreviating $N^X_n := |\mathcal{J}^X \cap (\delta n^2, \ceil{T n^2}  ]  |$, let $(j^n_i)_{1 \le i \le N^X_n }$ be the elements of $\mathcal{J}^X \cap (\delta n^2, \ceil{T n^2}  ]$ arranged in increasing order, set
\[j^n_{N^X_n +1} := \min \{i > \ceil{n^2 T} : i \in \mathcal{J}^X \},\]
and write $J^n:=\{j_i^n:\:i=1,\dots,N^X_n +1\}$.

\begin{proposition}[Jump-set spacing for $M^X$]
\label{prop:jumpsX}
As $n \to \infty$,
\[  N^X_n  \qquad \text{and} \qquad  \frac{{\rm{sep}} \left( J^n \right)}{n^2} \]
are respectively bounded above and bounded below in $\mathbb{P}_0$-probability.
\end{proposition}
\begin{proof}
Let $E = (E_i)_{i \in \mathbb{N}}$ be the sequence given by rearranging the elements of the trapping landscape $(\tau_x)_{x\in\mathbb{Z}}$ into the order that the relevant sites are visited by $S$, and let $\mathcal{J}$ be defined as in Section \ref{subsec:extremal} for the sequence $E$. Further denote by $(k_i)_{i\geq 1}$ the ordered list of elements in $\mathcal{J} \cap (n/C, \infty)$ and abbreviate $K_n := |\mathcal{J} \cap (n/C, nC]|$. Let $d_n$ be defined as in Section \ref{subsec:srw} for the simple random walk $S$. Note that, by Proposition~\ref{prop:diffdist}, for any $\varepsilon > 0$ there exists a $C > 0$ such that
\[ \mathbb{P}_0 \left(  d_{\delta n^2} > n/C \quad \text{and} \quad d_{ \ceil{Tn^2} } < n C \right) > 1 - \varepsilon,\]
which implies that
\[ \mathbb{P}_0 \left( N^X_n \le  \left| \mathcal{J} \cap (n/C, n C] \right|  \right) > 1 - \varepsilon.\]
Moreover, under $\mathbb{P}_0$, the sequence $E$ is i.i.d.\ with common distribution $\tau_0$, and is independent of $S$. Hence we can apply Proposition~\ref{prop:jumps} to bound
\[  | \mathcal{J} \cap (n/C, n C] |   \]
above in $\mathbb{P}_0$-probability, which proves the first result. Similarly, from Proposition~\ref{prop:diffdist} and the definition of ${\rm{sep}}(\cdot)$, it is possible to deduce that, for any $\varepsilon > 0$, there exists a $C > 0$ such that
\[  \mathbb{P}_0 \left(  {\rm{sep}}(J^n) \geq  {\rm{sep}}( \{ k_i : i = 1, 2, \ldots, K_n + 1 \})^2/ C \right)> 1 - \varepsilon.\]
Using the fact that $k_{K_n} \le nC$ and that $k_{K_n + 1}$ is either in $(nC, n(C+1)]$ or in $(n(C+1), \infty)$, we have the trivial bound
\[{\rm{sep}}( \{k_i^n:\:i=1,\dots,K_n+1\}) \geq \min\left\{{\rm{sep}}(  \mathcal{J} \cap (n/(C+1), n(C+1)]),n \right\},\]
and so Proposition~\ref{prop:jumps} applied to $\mathcal{J} \cap (n/(C+1), N(C+1)]$ gives the result.
\end{proof}

\begin{proposition}[Local time at deepest-visited traps]
\label{prop:jumplocaltime}
For each $1 \le i \le N^X_n$, let $\nu^i(k, 0)$ be defined similarly to $\nu(k, 0)$ in Section \ref{subsec:srw} for the simple random walk $(S_{k + j^n_i}-S_{j^n_i})_{k \in \mathbb{N}}$. Then, as $n \to \infty$,
\begin{align}
\label{eq:jumplocaltime1}
 \mathbb{P}_0\left( \nu^i(j^n_{i + 1} - j^n_i - 1, 0) > n / h_n \quad \text{for all } 1 \le i \le N^X_n \right)  \to 1
 \end{align}
and
\begin{align}
\label{eq:jumplocaltime2}
\mathbb{P}_0\left( \nu^i (\delta n^2 - 1, 0) > n / h_n \quad \text{for all } 1 \le i \le N^X_n \right) \to 1.
\end{align}
\end{proposition}
\begin{proof}
By the time-homogeneity of a SRW and the fact that $j_i^n$ is a stopping time for each $i$,
\[ \nu^i (n, 0) \stackrel{d}{=} \nu(n, 0),\]
and so it follows from Proposition~\ref{prop:ctslocaltime} that, as $n \to \infty$,
\[  \mathbb{P}_0\left( \nu^i (n^2 / h_n , 0) > n / h_n \right) = \mathbb{P}_0 \left( \nu (n^2 / h_n , 0)  > n / h_n \right)  \to 1. \]
Since, by Proposition~\ref{prop:jumpsX}, $N^X_n$ is bounded above in probability, it follows that
\[ \mathbb{P}_0\left( \nu^i (n^2 / h_n , 0) > n / h_n \quad \text{for each } 1 \le i \le N^X_n \right)  \to 1. \]
This is sufficient to establish equation \eqref{eq:jumplocaltime2} since $\nu^i (\cdot , 0)$ is non-decreasing. For equation \eqref{eq:jumplocaltime1}, simply apply the second part of Proposition~\ref{prop:jumpsX}, since $j^n_{i + 1} - j_i^n \ge {\rm{sep}}(J^n)$ for each $1\le i \le N^X_n$.
\end{proof}

In what follows, we make use of the product space $D(\mathbb{R}^+) \times D(\mathbb{R}^+)$. For a sequence of probability measures on $D(\mathbb{R}^+) \times D(\mathbb{R}^+)$, we denote by
\[\stackrel{J_1/J_1}{\Rightarrow}  \quad \text{and} \quad \stackrel{M_1/J_1}{\Rightarrow} \]
weak convergence of the first component in the $J_1$ and $M_1$ topologies respectively, and the simultaneous weak convergence of the second component in the $J_1$ topology.

\begin{proposition}[Functional limit theorem for the $S$-explored extremal and sum processes]
\label{prop:j1convX} Under $\mathbb{P}_0$, as $n \to \infty$,
\[\left(\frac{1}{n} L (M^X_{n^2 t}),\frac{1}{n} S_{n^2 t} \right)_{t\geq 0} \stackrel{J_1/J_1}{\Rightarrow} \left(m^B_t,B_t\right)_{t\geq 0},\]
and
\[\left( \frac{1}{n} L (\Sigma^X_{n^2 t} ),\frac{1}{n} S_{n^2 t} \right)_{t\geq 0} \stackrel{J_1/J_1}{\Rightarrow} \left(m^B_t,B_t\right)_{t\geq 0} \]
in distribution.
\end{proposition}
\begin{proof} Let $E = (E_i)_{i \in \mathbb{N}}$ be the sequence defined in the proof of Proposition~\ref{prop:jumpsX}. Define $d_n$ as in Section \ref{subsec:srw} for the simple random walk $S$, and also define the equivalent diffusion distance for the standard Brownian motion
\begin{equation}\label{dbtdef}
d^B_t:=\sup_{s\leq t}B_s-\inf_{s\leq t}B_s.
\end{equation}
Combining Proposition~\ref{prop:j1conv} and Donsker's invariance principle, we have under $\mathbb{P}_0$ that
\begin{equation}\label{previous}
\left( \frac{1}{n} L (M_{n t }), \frac{1}{n} S_{n^2 t}\right)_{t\geq 0} \stackrel{J_1/J_1}{\Rightarrow}\left( m_t,B_t\right)_{t\geq 0}
\end{equation}
in distribution, where $M_n$ and $m_t$ are defined as in Section \ref{subsec:extremal}. Note that, using the continuous mapping theorem (and the fact that $B$ is continuous almost-surely), we also have that $(n^{-1}d_{n^2t})_{t\geq 0}$ converges in distribution to $(d^B_t)_{t\geq 0}$ (in the $J_1$ topology). Together with the composition result of Lemma~\ref{lem:comp}, it follows that, under $\mathbb{P}_0$,
\begin{align}\label{aa1}
 \left(\frac{1}{n} L \left(M_{d_{n^2 t}}\right)\right)_{t\geq 0}
\stackrel{J_1}{\Rightarrow} \left(m_{d^B_t}\right)_{t\geq 0}
\end{align}
in distribution (simultaneously with the convergence at (\ref{previous})). Now, it is straightforward to check from the construction of the relevant processes that
\begin{equation}\label{aa2}
 \left(\frac{1}{n} L (M^X_{n^2 t} ),\frac{1}{n} S_{n^2 t}\right)_{t\geq 0}
\stackrel{d}{=}  \left(\frac{1}{n} L \left(M_{d_{n^2 t}}\right),\frac{1}{n} S_{n^2 t}\right)_{t\geq 0} .
\end{equation}
Moreover, we have that
\begin{equation}\label{aa3}
 \left(m_{d^B_t},B_t\right)_{t\geq 0}
\stackrel{d}{=}  (m^B_t,B_t)_{t \ge 0}.
\end{equation}
Indeed, by conditioning on $B$ and applying the spatial homogeneity of the underlying point process, checking that the finite dimensional distributions of the two above processes agree is easy, and (\ref{aa3}) follows readily from this. Putting (\ref{aa1}), (\ref{aa2}) and (\ref{aa3}) together completes the proof of the first claim of the proposition. The proof of the second claim is similar.
\end{proof}

\begin{corollary}[Lower bound for the $S$-explored extremal and sum processes]
\label{cor:maxX}
As $n \to \infty$,
\[ \mathbb{P}_0\left(   M^X_{n^2 \delta } \ge L^{-1}(n  / h_n)  \right) \to 1 \quad \text{and} \quad \mathbb{P}_0\left(   \Sigma^X_{n^2 \delta } \ge L^{-1}(n  / h_n)  \right) \to 1.\]
\end{corollary}
\begin{proof}
By the existence of the scaling limit, as $n \to \infty$,
\[ \mathbb{P}_0\left(  \frac{1}{n} L( M^X_{n^2 \delta } ) > 1/h_n  \right) \to 1 \quad \text{and} \quad \mathbb{P}_0\left(   \frac{1}{n} L ( \Sigma^X_{n^2 \delta } ) > 1/h_n  \right) \to 1 \]
both hold. The result then follows by the definition of $L^{-1}$.
\end{proof}

\subsection{Squeezing the clock process}
The next step is to show that, under suitable rescaling, the clock process $A$ is squeezed (with high probability) between the extremal and sum processes $M^X$ and $\Sigma^X$; the squeezing is done in both time and space.

\begin{proposition}
As $n \to \infty$,
\begin{align}
\label{eq:bounds1}
\mathbb{P}_0 \left( \frac{1}{n} A_{n^2 t} < \Sigma^X_{n^2 t} \, h_n   \quad \text{for all } t \in [\delta, T] \,  \right) \to 1.
\end{align}
Moreover, for each $t \in [\delta, T]$ and $n \ge 0$ there exists a $\mathbb{P}_0$-measurable random time $s^n_t \in [t, t + \delta]$ such that, as $n \to \infty$,
\begin{align}
\label{eq:bounds2}
\mathbb{P}_0 \left(  \frac{1}{n} A_{n^2 s^n_t} > M^X_{n^2 s^n_t} / h_n \quad \text{for all } t \in [\delta, T] \, \right)\to 1.
\end{align}
\end{proposition}
\begin{proof}
Consider first the limit at \eqref{eq:bounds1}. Let $\nu_{\text{max}}(n)$ be defined as in Section \ref{subsec:srw} for the simple random walk $S$. Then, by the definition of $A_n$ and $\Sigma^X_n$,
\[ A_{n} \le \nu_{\text{max}}(n) \, \Sigma^X_{n},\]
for all $n \ge 0$, and so
\[ \frac{1}{n} A_{n^2 t} \le \frac{1}{n} \nu_{\text{max}}(n^2 T) \, \Sigma^X_{n^2 t}  \quad \text{for all } t \in [\delta, T], \]
since $\nu_{\text{max}}(\cdot)$ is non-decreasing. Equation \eqref{eq:bounds1} then follows by applying Proposition~\ref{prop:ctslocaltime}.

We now work towards equation \eqref{eq:bounds2}, starting with an explicit construction of $s^n_t$ on the event
\[\mathcal{A}_{n,\delta_1}:=\{\mathcal{J}^X\cap(\delta_1 n^2,\delta n^2]\neq \emptyset\},\]
for each $n \ge 0$ and $\delta_1 \in (0, \delta]$. To this end, let $(j^{n, \delta_1}_i)_{i=1}^{N^X}$ be the elements of the set $\mathcal{J}^X\cap(\delta_1 n^2,\lceil T n^2\rceil]$ arranged in increasing order. Note that, for simplicity, in what follows we will suppress the dependence of $j^{n, \delta_1}_i$ on $n$ and $\delta_1$. For any $t \in [\delta, T]$ let $i_t$ be the index of the last jump $j_i$ strictly less than $n^2 t+1$, that is,
\[ i_t := \max \{  1 \le i \le N^X : j_i < n^2 t + 1 \}. \]
Then, define $s^n_t$ by
\[   s^n_t := \min \left\{   \frac{1}{n^2} \left( j_{i_t + 1} - 1 \right)   , \, t + \delta \right\}. \]
We note that by the monotonicity of the events $\mathcal{A}_{n,\delta_1}$, the above construction well-defines $s_t^n$ on the whole of $\mathcal{A}_{n}:=\cup_{\delta_1\leq\delta}\mathcal{A}_{n,\delta_1}$. Furthermore, by arbitrarily extending the definition of $s_t^n$ by setting $s^n_t=t$ for $t\in[\delta,T]$ on the event $\mathcal{A}_{n}^c$, we ensure that $s^n_t$ is $\mathbb{P}_0$-measurable. We clearly also have that $s^n_t \in [t, t + \delta]$. Finally, this construction also guarantees that, on $\mathcal{A}_{n,\delta_1}$, for each $t \in [\delta, T]$,
\begin{align}
\label{eq:bounds3}
i_{s_t} = i_t
\end{align}
and moreover that
\begin{align}
\label{eq:bounds4}
n^2 s^n_t  - j_{i_t}  \ge \min \left\{ j_{i_t + 1} - j_{i_t} -1 \ , \ \delta n^2 -1 \right\}.
\end{align}
Recalling the definition of $\nu^i(n, 0)$ from Proposition~\ref{prop:jumplocaltime} (substituting $\delta_1$ for $\delta$), we have by the definition of $A_n$ and $M^X_n$ that, on $\mathcal{A}_{n,\delta_1}$,
\[ A_{n^2 t} \ge \nu^{i_t}(n^2 t - j_{i_t}, 0) \, M^X_{n^2 t} \]
for each $t \in [\delta, T]$. Combining this with equations \eqref{eq:bounds3} and \eqref{eq:bounds4} gives, on $\mathcal{A}_{n,\delta_1}$,
\[ \frac{1}{n} A_{n^2 s_t^n} \ge \frac{1}{n} \nu^{i_t} \left( \min \left\{ j_{i_t + 1} - j_{i_t} -1 , \delta n^2 - 1 \right\} , 0 \right) \, M^X_{n^2 s_t^n}  , \]
and so Proposition~\ref{prop:jumplocaltime} yields that, for any $\delta_1\leq \delta$
\[\liminf_{n\to\infty}\mathbb{P}_0 \left(  \frac{1}{n} A_{n^2 s^n_t} > M^X_{n^2 s^n_t} / h_n \quad \text{for all } t \in [\delta, T] \, \right)\geq 1-\limsup_{n\to\infty} \mathbb{P}_0\left(\mathcal{A}_{n,\delta_1}^c\right).\]
Finally, we have that
\[\mathbb{P}_0\left(\mathcal{A}_{n,\delta_1}\right)=\mathbb{P}_0\left(M^X_{\delta n^2}> M^X_{\delta_1 n^2}  \right)\geq\mathbb{P}_0\left(n^{-1}L(M^X_{\delta n^2})>n^{-1}L(M^X_{\delta_1 n^2})\right).\]
By Proposition~\ref{prop:j1convX}, the liminf as $n\to\infty$ of the right-hand side above is bounded below by $\mathbb{P}_0(m^B_{\delta}>m^B_{\delta_1})=1-\delta_1/\delta$, or to put this another way
\[\limsup_{n\to\infty} \mathbb{P}_0\left(\mathcal{A}_{n,\delta_1}^c\right)\leq \frac{\delta_1}{\delta},\]
which can be made arbitrarily small by adjusting the choice of $\delta_1$.
\end{proof}

\begin{proposition}
\label{prop:squeeze}
As $n \to \infty$,
\begin{align}
\label{eq:squeeze1}
\mathbb{P}_0 \left( \frac{1}{n} L \left( \frac{1}{n} A_{n^2 t} \right) < \frac{1}{n} L\left( \Sigma^X_{n^2 t} \right) + \delta  \quad \text{for all } t \in [\delta, T] \, \right) \to 1.
\end{align}
Moreover, for each $t \in [\delta, T]$ and $n \ge 0 $ there exists a $\mathbb{P}_0$-measurable random time $s^n_t \in [t, t + \delta]$ such that, as $n \to \infty$,
\begin{align}
\label{eq:squeeze2}
\mathbb{P}_0 \left(\frac{1}{n} L \left( \frac{1}{n} A_{n^2 s^n_t} \right) >  \frac{1}{n} L \left( M^X_{n^2 s^n_t}  \right) - \delta  \quad \text{for all } t \in [\delta, T] \,  \right)  \to 1.
\end{align}
\end{proposition}
\begin{proof}
Consider first equation \eqref{eq:squeeze1}. Starting from equation \eqref{eq:bounds1}, applying $L$ to both sides of the inequality and then dividing by $n$ we get that, as $n \to \infty$,
\[\mathbb{P}_0 \left(  \frac{1}{n} L \left( \frac{1}{n} A_{n^2 t} \right) \le  \frac{1}{n} L \left( \Sigma^X_{n^2 t} \, h_n \right)  \quad \text{for all } t \in [\delta, T] \, \right) \to 1.\]
Note that by Corollary \ref{cor:maxX}, and since $\Sigma^X_n$ is non-decreasing, as $n \to \infty$,
\[ \mathbb{P}_0 \left(  \Sigma^X_{n^2 t} >  L^{-1}(n  / h_n )    \quad \text{for all } t \in [\delta, T] \,  \right) \to 1.\]
By equation \eqref{eq:decay1}, this means that for arbitrary $\eta>0$, as $n \to \infty$,
\[ \mathbb{P}_0  \left( \frac{1}{n} L \left( \Sigma^X_{n^2 t}  \,  h_n \right) <  \frac{1}{n} L \left( \Sigma^X_{n^2 t} \right)(1 + \eta) \quad \text{for all } t \in [\delta, T] \,   \right) \to 1.\]
Since we have from Proposition~\ref{prop:j1convX} that
\[\mathbb{P}_0  \left(  \frac{\eta}{n} L \left( \Sigma^X_{n^2 T} \right)\geq \delta\right)\to
 \mathbb{P}_0  \left(  \eta m_T^B\geq \delta\right),\]
and the right-hand side converges to 0 as $\eta\to 0$, this is enough to yield the result.

Consider then equation \eqref{eq:squeeze2}. Similarly, equation \eqref{eq:bounds2} gives that, as $n \to \infty$,
\[  \mathbb{P}_0 \left( \frac{1}{n} L \left( \frac{1}{n} A_{n^2 s^n_t} \right) > \frac{1}{n} L \left( M^X_{n^2 s^n_t} / h_n \right)   \quad \text{for all } t \in [\delta, T] \, \right) \to 1.\]
As before, Corollary \ref{cor:maxX}, equation \eqref{eq:decay1} and Proposition~\ref{prop:j1convX} then imply the result.
\end{proof}

Under Assumption \ref{assumpt}, we establish the stronger uniform convergence (in space) of $A$ to $\Sigma^X$.

\begin{proposition}
\label{prop:assumptsqueeze}
Under Assumption \ref{assumpt}, as $n \to \infty$,
\[ \sup_{ t \in [\delta, T]} \left| \frac{1}{n} L \left(A_{n^2 t} \right)- \frac{1}{n} L \left( \Sigma^X_{n^2 t} \right) \right|  \to 0 \quad \text{in } \mathbb{P}_0\text{-probability.}  \]
\end{proposition}

\begin{proof}
Assume that $h_n \to \infty$ is growing sufficiently slowly that equation \eqref{eq:decay2} holds and let $\nu_{\text{min}}(n)$ be defined as in Section \ref{subsec:srw} for the simple random walk $S$. Then, by definition, $A_n \ge \nu_{\text{min}}(n) \Sigma_n^X$, for all $n \ge 0$. Since, by Proposition~\ref{prop:ctslocaltimemin}, as $n \to \infty$,
\[  \mathbb{P}_0 \left(  \nu_{\text{min}}(n^2 t)  > 1/(n h_n)   \quad \text{for all } t \in [\delta, T] \, \right) \to 1, \]
together with equation (\ref{eq:bounds1}), we have that, as $n \to \infty$,
\[ \mathbb{P}_0 \left( \frac{1}{n} L \left( \Sigma^X_{n^2 t} / (n h_n) \right) \le \frac{1}{n} L \left( A_{n^2 t} \right) \le \frac{1}{n} L \left(\Sigma^X_{n^2 t} \, (n h_n) \right) \quad \text{for all } t \in [\delta, T] \, \right) \to 1.  \]
Finally, as in the proof of Proposition~\ref{prop:squeeze}, Corollary \ref{cor:maxX} and equation \eqref{eq:decay2} then jointly imply that for any $\eta>0$, as $n \to \infty$,
\[   \mathbb{P}_0 \left( \frac{1}{n} L \left( \Sigma^X_{n^2 t}  / (n h_n) \right) >  \frac{1}{n} L \left( \Sigma^X_{n^2 t} \right)(1 - \eta) \quad \text{for all } t \in [\delta, T] \, \right) \to 1 \]
and
\[   \mathbb{P}_0 \left( \frac{1}{n} L \left(  \Sigma^X_{n^2 t}  \, (n h_n) \right) <  \frac{1}{n} L \left( \Sigma^X_{n^2 t} \right)(1 + \eta) \quad \text{for all } t \in [\delta, T] \, \right) \to 1.\]
By applying Proposition~\ref{prop:j1convX}, it follows that for any $\eta,\varepsilon>0$, as $n\rightarrow\infty$
\[\limsup_{n\rightarrow\infty}\mathbb{P}_0\left(\sup_{ t \in [\delta, T]} \left| \frac{1}{n} L \left(A_{n^2 t} \right)- \frac{1}{n} L \left( \Sigma^X_{n^2 t} \right) \right|\geq \varepsilon\right)\leq \mathbb{P}_0\left(\eta m_T^B\geq \varepsilon\right).\]
Letting $\eta\to 0$ completes the proof.
\end{proof}

\subsection{Proofs of the main convergence results}

We are now ready to prove the main result of this section, from which the the conclusions stated in the introduction follow easily.

\begin{proposition}[Restated functional limit theorems for the clock process]
\label{prop:fltA2} Under $\mathbb{P}_0$, as $n\to \infty$,
\[ \left(\frac1n L\left(\frac1n A_{n^2t}\right) , \, \frac1n S_{n^2t} \right)_{t\geq 0}\stackrel{M_1/J_1}{\Rightarrow} \left(m^B_t  , B_t \right)_{t\geq 0}\]
in distribution. Moreover, if Assumption \ref{assumpt} holds, then under $\mathbb{P}_0$, as $n \to \infty$,
\[ \left(\frac1n L(A_{n^2t}) , \,\frac1n S_{n^2t}  \right)_{t\geq0} \stackrel{J_1/J_1}{\Rightarrow} \left(m^B_t, B_t \right)_{t\geq 0} \]
in distribution.
\end{proposition}

\begin{proof} Recalling Proposition~\ref{prop:j1convX} and the bounds in Proposition~\ref{prop:squeeze}, the first statement follows from the convergence result of Lemma~\ref{lem:sconv2}. Similarly, recalling Proposition~\ref{prop:j1convX} and the bounds in Proposition~\ref{prop:assumptsqueeze}, the second statement follows from the convergence result of Lemma~\ref{lem:sconv1}.
\end{proof}

\begin{proof}[Proof of Theorem~\ref{thm:fltA} and Theorem~ \ref{thm:assumptA}]
The conclusions of Theorems \ref{thm:fltA} and \ref{thm:assumptA} follow immediately from the previous result.
\end{proof}

To complete this section, we derive the convergence of the BTM $X$ to the limit process $B_{I^B}$, as stated in Theorems \ref{thm:fltX} and \ref{thm:assumptX}. The bulk of the work has already been done in establishing the convergence of the clock process above; only technicalities involving convergence results for the various topologies remain.

\begin{proof}[Proof of Theorem~\ref{thm:fltX} and Theorem~\ref{thm:assumptX}]
Since the right-continuous inverse of the process $n^{-1} L(n^{-1} A_{n^{-2}t})$ is given by $n^{-2} I^S_{n L^{-1}(nt) }$, applying the inversion result of Lemma~\ref{lem:inv} to Proposition~\ref{prop:fltA2} yields that under $\mathbb{P}_0$, as $n\to \infty$,
\begin{equation}\label{conc1}
\left(n^{-2} I^S_{ n L^{-1}(nt) } , \, n^{-1} S_{n^2t} \right)_{t\geq 0}\stackrel{M_1/J_1}{\Rightarrow} \left(I^B_t  , B_t \right)_{t\geq 0}
\end{equation}
in distribution. Similarly, noting that the right-continuous inverse of
$n^{-1}L( A_{n^2t}) $ is $n^{-2} I^S_{ L^{-1}(nt) }$, we argue similarly to deduce that if Assumption \ref{assumpt} holds, then under $\mathbb{P}_0$, as $n\to \infty$,
\begin{equation}\label{conc2}
\left(n^{-2} I^S_{ L^{-1}(nt) }  , \, n^{-1} S_{n^{2}t}  \right)_{t\geq0} \stackrel{M_1/J_1}{\Rightarrow} \left(I^B_t, B_t \right)_{t\geq 0}
\end{equation}
in distribution. Consequently, recalling that the law of $X$ under $\mathbb{P}_0$ is identical to that of $S_{I^S}$, the second statement of Lemma~\ref{lem:comp} allows us to deduce the desired results by composing the two coordinates of (\ref{conc1}) and (\ref{conc2}).
\end{proof}

\section{The extremal FIN process}
\label{sec:alpha0}
In this section we prove that the scaling limit $B_{I^B}$ is the natural analogue of the FIN diffusion with parameter $\alpha \in (0, 1)$ in the limiting case $\alpha = 0$. In particular, we prove Theorem~\ref{thm:alpha0}.

\begin{lemma}[Sum-to-max]
\label{lem:sum-to-max}
Let $(c_i, v_i)_{i \in \mathbb{N}}$ be a set of points in $\mathbb{R}^+ \times (0,\infty)$ with the property that, for each $s \in (1, \infty)$,
\[  \sum_i c_i v_i^s < \infty.  \]
Then, as $s \to \infty$,
\[  \left( \sum_i c_i v_i^s \right)^{1/s} \to \sup_{i : c_i > 0}  v_i  \]
\end{lemma}
\begin{proof}
Define the function $v: \mathbb{N} \to \mathbb{R}^+$ by the map $i \mapsto v_i$ and denote by $\mu$ the (possibly infinite) measure
\[ \mu := \sum_i c_i \delta_i. \]
Then the claim is just the fact that the $L_s$ norm of $v$ with respect to the measure $\mu$ converges, if finite, to the $L_\infty$ norm of $v$ with respect to $\mu$ (see, for example, \cite[Section 2.1]{LL01}).
\end{proof}

Recall the definitions of the processes $m^B$ and $m^{B, \alpha}$ from (\ref{mtdef}) and (\ref{mtadef}), which are the clock-processes for the extremal FIN process and the FIN diffusion with parameter $\alpha$, respectively.

\begin{proposition}[Convergence of the clock-processes]
\label{prop:alpha0clock}
As $\alpha \to 0$,
\[  \left( \left( m^{B, \alpha}_t \right)^\alpha \right)_{t \ge 0} \stackrel{M_1}{\rightarrow} \left( m^B_t \right)_{t\geq0} \]
where $\stackrel{M_1 }{\rightarrow}$ denotes convergence in the $M_1$ topology, almost-surely with respect to the joint law of $\mathcal{P}$ and $B$.
\end{proposition}
\begin{proof} We start by proving convergence for a fixed $t$. By definition, we have that
\[ m^{B, \alpha}_t :=  \sum_{i: L_t(x_i) > 0} L_t(x_i) v^{1/\alpha}_i . \]
It is an elementary exercise to deduce from this, the fact that $\sup_{x\in \mathbb{R}}{L_t(x)}<\infty$ and $d_t^B<\infty$ almost-surely (where $d_t^B$ was defined at (\ref{dbtdef})), and the definition of $\mathcal{P}$, that $m^{B, \alpha}_t$  is finite for any $\alpha\in(0,1)$, almost-surely. We also have the identity
\[ m^B_t = \sup_{i: L_t(x_i) > 0}  v_i  \]
almost-surely. Indeed, $L_t(x)>0$ if and only if $x\in(\inf_{s\in[0,t]}B_s,\sup_{s\in[0,t]}B_s)$ almost-surely (see, for example, \cite[Corollary 22.18]{Kall}). Moreover, we may assume that there are no points $(x_i,v_i)$ in $\mathcal{P}$ with $x_i\in\{\inf_{s\in[0,t]}B_s,\sup_{s\in[0,t]}B_s\}$ almost-surely. Hence, applying Lemma~\ref{lem:sum-to-max} to the set of points $(L_t(x_i), v_i))_{i \in \mathbb{N}} \in \mathbb{R}^+ \times(0,\infty)$ yields that, for each fixed $t$, as $\alpha \to 0$,
\[  \left( m^{B, \alpha}_t \right)^\alpha \to  m^B_t \]
almost-surely. By countability, we immediately deduce that this convergence holds for all rational times simultaneously. As the process $m^{B, \alpha}$ is non-decreasing for each $\alpha$ by definition, almost-sure convergence in $M_1$ follows.
\end{proof}

\begin{remark}
That the convergence in Proposition~\ref{prop:alpha0clock} does not hold in the stronger $J_1$ topology can be easily seen from the fact that $m^{B, \alpha}$ is continuous for each $\alpha$ whereas the limit process $m^{B}$ is not continuous.
\end{remark}

\begin{proof}[Proof of Theorem~\ref{thm:alpha0}]
The proof of Theorem~\ref{thm:alpha0} follows from Proposition~\ref{prop:alpha0clock} in an identical manner to in the proofs of Theorems \ref{thm:fltX} and \ref{thm:assumptX}, by applying the inversion and composition results of Lemmas \ref{lem:inv} and \ref{lem:comp}.
\end{proof}

\begin{remark} \label{alpha1rem}
A result corresponding to Theorem~\ref{thm:alpha0} can be also established for FIN diffusions in the $\alpha\to 1^-$ limit. In particular, we claim that as $\alpha \to 1^-$,
\begin{equation}\label{claim}
\left(B_{ I^{B, \alpha}_{(1-\alpha)^{-1}t}} \right)_{t \ge 0} \stackrel{U}{\rightarrow} \left( B_{t} \right)_{t\geq0},
\end{equation}
where $\stackrel{U}{\rightarrow}$ denotes uniform convergence over compact time intervals, almost-surely with respect to the joint law of $\mathcal{P}$ and $B$. Since it is not directly related to the main results of this paper, we only sketch a proof. Defining $(x_i,v_i)_{i\in\mathbb{N}}$ as before and setting $\Sigma:=\sum_{i:x_i\in[0,1],v_i\leq 1}v_i^{1/\alpha}$, it is an elementary exercise to check that
\[\mathbf{E}\left(\Sigma\right)=\frac{\alpha}{1-\alpha},\qquad\mathrm{Var}\left(\Sigma\right)=\frac{\alpha}{2-\alpha}.\]
(One can do so using Campbell's theorem for Poisson point processes, for example.) Consequently,
\[\mathbf{P}\left(\left|(1-\alpha)\Sigma-\alpha\right|\geq \varepsilon\right)\leq \frac{\alpha(1-\alpha)^2}{\varepsilon^2(2-\alpha)},\]
and a Borel-Cantelli argument yields
\begin{equation}\label{firstlim}
(1-\alpha)\Sigma\rightarrow 1,
\end{equation}
along the subsequence $\alpha=1-n^{-1}$, almost-surely. By the monotonicity of $\Sigma$ in $\alpha$, this is readily extended to almost-sure convergence as $\alpha \to 1^-$. From this, we deduce that
\begin{equation}\label{secondlim}
(1-\alpha)\sum_{i:x_i\in[a,b]}v_i^{1/\alpha}\rightarrow (b-a),\qquad\mbox{as }\alpha \to 1^-,\:\forall a\leq b,
\end{equation}
almost-surely (adding the finite number of terms with $v_i>1$ clearly does not affect the limit at (\ref{firstlim}), and then a countability argument and monotonicity can be used to establish (\ref{secondlim})). We note that the convergence at (\ref{secondlim}) implies almost-sure vague convergence of the measures $(1-\alpha)\sum_{i}\delta_{x_i}v_i^{1/\alpha}$ to the Lebesgue measure on the real line. Thus, using the continuity of the Brownian local times, we obtain
\[(1-\alpha)m_t^{B,\alpha}=(1-\alpha)\sum_{i}L_t(x_i)v_i^{1/\alpha}\rightarrow \int L_t(x)dx = t\]
uniformly over compact intervals of $t$, almost-surely. The claim at (\ref{claim}) then follows by taking inverses and composing with $B$, similarly to the proof of Theorem~\ref{thm:alpha0}.
\end{remark}

\section{Transparent traps}
\label{sec:trans}

In this section we establish the scaling limits for the $\beta$-transparent BTM of Theorem~\ref{thm:tt}. We achieve this by verifying the sufficient conditions given in \cite{BenArous13} for the convergence of randomly trapped random walks to the standard Brownian motion and the FK process respectively.

We begin by proving a consequence of the second-order slow-variation of $L$ for certain expectations involving $\tau_0$; the spirit is similar to that of de Haan's theorem, see \cite[Section 3.7]{Bingham87}.

\begin{proposition}
\label{prop:sosv} Assume that $L$ is second-order slowly varying, i.e.\ that it satisfies (\ref{sosv}). Let $f:\mathbb{R}^+ \to \mathbb{R}^+$ be a continuously differentiable function such that $f(t) \to 0$ as $t \to \infty$. Moreover, suppose that there exists a $\delta > 0$ for which $f(t)=o(t^\delta)$ as $t\rightarrow 0$, and both $f'(t) t^\delta$ and $f'(t) t^{-\delta}$ are integrable. Then the function
\[ \Gamma(n) := \mathbf{E}\left[ f(\tau_0 / n) \right]\]
satisfies
\[\lim_{n \to \infty} \frac{L(n) \Gamma(n)}{g(n)} = - \lambda  \int_0^\infty f'(t) \log t \, dt\]
for some constant $\lambda\neq 0$ that only depends on $L$ and $g$. In particular, if the integral on the right-hand side is non-zero (note that the assumptions ensure that it is finite), then $\Gamma$ is slowly varying.
\end{proposition}
\begin{proof} By standard de Haan theory, the second-order slowly varying property implies that $g(u)$ is regularly varying with a non-positive index (see the discussion at the start of \cite[Section 3.12]{Bingham87}). Applying this fact and the divergence of $L(u)$, it can be deduced from the representation theorem of \cite[Theorem 3.12.2]{Bingham87} that $g$ is actually slowly varying, and further that $k(v) = \lambda \log v$ for some constant $\lambda \neq 0$. Moreover, it is easy to see that
\[ \bar{F}(u) := \PP(\tau_0 > u) = 1/L(u) \]
is also second-order slowly varying, satisfying
\begin{align}
\label{eq:sosv2}
\lim_{u \to \infty} \frac{ \bar{F}(u) - \bar{F}(uv)}{\bar{F}(u) g(u)} = - k(1/v) = \lambda \log v, \quad v > 0,
\end{align}
for the same $g, k$ and $\lambda$. We therefore have the following uniform bound (see \cite[(3.1.5)]{Bingham87}): for any $\delta>0$, there exist constants $K, n_0 > 0$ such that
\begin{align}
\label{eq:sosv3}
\left| \frac{  \bar{F}(n ) - \bar{F}(nt)}{\bar{F}(n) g(n)}\right| \le K t^\delta \quad \text{for all } n > n_0,\:t\geq 1.
\end{align}
Suppose $t\in[n_0/n,1]$, where $n\geq n_0$. Setting $m=nt\geq n_0$, (\ref{eq:sosv3}) implies
\[ \left|\frac{  \bar{F}(n ) - \bar{F}(nt)}{\bar{F}(n) g(n)}\right|\leq Kt^{-\delta}\left|\frac{\bar{F}(m) g(m)}{\bar{F}(m/t) g(m/t)}\right|.\]
Since $\bar{F}g$ is slowly varying, we may bound $|{\bar{F}(m) g(m)}/{\bar{F}(m/t) g(m/t)}|$ by $K't^{-\delta}$. Reparameterising $\delta$, this yields that: for any $\delta>0$, there exist constants $K, n_0 > 0$ such that
\begin{align}
\label{eq:sosv4}
\left| \frac{  \bar{F}(n ) - \bar{F}(nt)}{\bar{F}(n) g(n)}\right| \le K \max\{t^\delta,t^{-\delta}\} \quad \text{for all } n > n_0,\:t\geq n_0/n.
\end{align}

For the remainder of the proof, we choose $\delta$ such that the assumptions on $f$ are satisfied, and select $K,n_0$ such that (\ref{eq:sosv4}) holds. Partitioning the integral and integrating by parts, we obtain
\begin{align*}
   \Gamma(n)  &= \mathbf{E}\left[ f(\tau_0 / n)\mathbf{1}_{\{\tau_0\geq n_0\}} \right]+\mathbf{E}\left[ f(\tau_0 / n)\mathbf{1}_{\{\tau_0<n_0\}} \right]\\
   & =   \left[  f(t) (1 - \bar{F}(nt) ) \right]_{n_0/n}^\infty  - \int_{n_0/n}^\infty f'(t) ( 1 - \bar{F}(nt)) \, dt  + T_1\\
  & = -\int_{n_0/n}^\infty f'(t) \left( \bar{F}(n) - \bar{F}(nt)\right) dt +T_1+T_2+T_3,
\end{align*}
where
\begin{align*}
T_1& :=\mathbf{E}\left[ f(\tau_0 / n)\mathbf{1}_{\{\tau_0<n_0\}} \right],\\
T_2& := -f(n_0/n)(1-\bar{F}(n_0)),\\
T_3& := - (1 - \bar{F}(n)) \int_{n_0/n}^\infty f'(t)  dt= (1 - \bar{F}(n)) f(n_0/n).
\end{align*}
Since $f(t)=o(t^\delta)$ as $t\rightarrow 0$ and $\bar{F}(n)g(n)$ is slowly varying, it holds that
\begin{align*}
\left|\frac{T_i}{\bar{F}(n)g(n)}\right|&\leq \left|\frac{\sup_{t\leq n_0/n}\left|f(t)\right|}{\bar{F}(n)g(n)}\right|\rightarrow 0
\end{align*}
as $n\to\infty$, for $i=1,2,3$. Thus, since both $f'(t) t^\delta$ and $f'(t) t^{-\delta}$ are integrable by assumption, the bound at (\ref{eq:sosv4}) allows us to apply the dominated convergence theorem to deduce that
\begin{align*}
 \lim_{n \to \infty} \frac{L(n) \Gamma(n)}{g(n)}
& =-\lim_{n \to \infty} \int_{0}^\infty f'(t) \frac{\bar{F}(n) - \bar{F}(nt)}{\bar{F}(n) g(n)}\mathbf{1}_{\{t\geq n_0/n\}} dt\\
  &= -  \int_0^\infty f'(t) \lim_{n \to \infty}  \frac{ \bar{F}(n) - \bar{F}(nt)}{\bar{F}(n) g(n)} \, dt \\
& = - \lambda  \int_0^\infty f'(t) \log t \, dt,
\end{align*}
where we recall the limit at (\ref{eq:sosv2}) to deduce the final, desired equality. Given that $L$ and $g$ are both slowly varying, and the class of slowly varying functions is closed under division, the second statement of the Proposition~readily follows.
\end{proof}

In the next result we derive asymptotic properties of the Laplace transform of $\pi_0$, the law of a holding time at $0$ conditional on $\tau_0$. In the proof, we will denote the law of $\tau_0$ under $\mathbf{P}$ by $\nu$.

\begin{proposition}
\label{prop:lp}
Assume $L$ is second-order slowly varying and let $\beta \in (0, 1)$. Define
\[ \hat{\pi}_0(\varepsilon) := \int_{0}^\infty e^{-\varepsilon t}\pi_0(dt).\]
Then the function
\[ \Gamma_1(\varepsilon) := \mathbf{E} \left[  1 -  \hat{\pi}_0(\varepsilon) \right] \]
satisfies
\[ \lim_{\varepsilon \to 0} \varepsilon^{-\beta} \bar{\ell}(\varepsilon^{-1})^{-1} \Gamma_1(\varepsilon)  = 1 \]
for the slowly varying function $ \bar{\ell}(n) := g(n)/L(n)$. Moreover, the function
 \[ \Gamma_2(\varepsilon) := \mathbf{E} \left[ \left(  1 - \hat{\pi}_0\left(\Gamma_1^{-1}(\varepsilon^2)\right)  \right)^2 \right]  \]
 satisfies
\[ \lim_{\varepsilon \to 0} \varepsilon^{-3} \Gamma_2(\varepsilon)  = 0, \]
where $\Gamma_1^{-1}(\varepsilon) := \inf \{s : \Gamma(s) > \varepsilon \}$ denotes the right-continuous inverse of $\Gamma_1$.
\end{proposition}
\begin{proof} On $\tau_0\geq 1$, we have that
\[\hat{\pi}_0(\varepsilon)=\frac{1}{1 + \varepsilon} - \tau_0^{-\beta} \left( \frac{1}{1 + \varepsilon} - \frac{1}{1 + \varepsilon \tau_0} \right). \]
Otherwise,
\[\hat{\pi}_0(\varepsilon)= \frac{1}{1 + \varepsilon \tau_0}. \]
Consequently, as $\varepsilon\to 0$,
\begin{eqnarray*}
  \Gamma_1(\varepsilon)  &=&\int_0^1 \frac{\varepsilon t}{1 + \varepsilon t}\nu(dt)
+\int_1^\infty \left(\frac{\varepsilon }{1 + \varepsilon }+t^{-\beta} \left( \frac{1}{1 + \varepsilon} - \frac{1}{1 + \varepsilon t} \right)\right) \, \nu(dt) \\
&=&\int_1^\infty t^{-\beta} \left( \frac{ \varepsilon t}{1 + \varepsilon t} -\frac{\varepsilon}{1 + \varepsilon} \right) \, \nu(dt)+O(\varepsilon)\\
&=&\int_0^\infty \frac{ \varepsilon t^{1-\beta}}{1 + \varepsilon t}\nu(dt)+O(\varepsilon).
\end{eqnarray*}
After the change of variables $s = \varepsilon t$, this gives, as $\varepsilon \to 0$,
\[ \Gamma_1(\varepsilon) =   \varepsilon^{\beta} \int_0^\infty \frac{s^{1-\beta}}{1 + s} \, \nu(\varepsilon^{-1} ds)+O(\varepsilon)= \varepsilon^{\beta}\mathbf{E}\left(f_1(\varepsilon \tau_0)\right)+O(\varepsilon),\]
where
\[f_1(s) :=  \frac{s^{1-\beta}}{1 + s} .\]
It is easy to check that the conditions of Proposition~\ref{prop:sosv} are satisfied for $f_1$ (with $\delta < \min\{ \beta, 1 - \beta\}$). Moreover, integration by parts yields
\[\int_0^\infty f_1'(s)\log s ds = -\int_0^\infty \frac{f_1(s)}{s}ds <0,\]
and so an application of Proposition~\ref{prop:sosv} gives the first statement.

By the conclusion of the previous paragraph, we have that $\Gamma_1(\varepsilon)\sim \varepsilon^{\beta} \bar{\ell}(\varepsilon^{-1})$, and so $\Gamma_1^{-1}(\varepsilon^2)\sim \varepsilon^{2/\beta} {\ell}(\varepsilon^{-1})^{-1}$
for some slowly varying function $\ell$ (as $\varepsilon \to 0$). In particular, for any $\delta>0$, there exist $c,\varepsilon_0>0$ such that $\Gamma_1^{-1}(\varepsilon^2)\leq c\varepsilon^{\gamma}$ for every $\varepsilon\leq \varepsilon_0$, where $\gamma:=\frac{2}{\beta}-\delta$. Applying this bound and arguing similarly to above, we deduce that, as $\varepsilon \to 0$,
\begin{eqnarray*}
 \Gamma_2(\varepsilon^{1/\gamma}) &\leq & \mathbf{E} \left[ \left(  1 - \hat{\pi}_0\left(c\varepsilon\right)  \right)^2 \right]\\
 &=&  (c\varepsilon)^{2\beta} \int_0^\infty \frac{s^{2-2\beta}}{(1 + s)^2} \, \nu(\varepsilon^{-1} ds)+ O(\varepsilon^{1+\beta-\delta})\\
 &=&
 (c\varepsilon)^{2\beta}\mathbf{E}\left(f_2(\varepsilon \tau_0)\right)+O(\varepsilon^{1+\beta-\delta}),
 \end{eqnarray*}
where
\[f_2(s) :=  \frac{s^{2-2\beta}}{(1 + s)^2}. \]
Since $f_2$ also satisfies the conditions of Proposition~\ref{prop:sosv}, we obtain from this that $\Gamma_2(\varepsilon^{1/\gamma})=O(\varepsilon^{2\beta-\delta})$ as $\varepsilon \to 0$, and the result follows.
\end{proof}

\begin{proof}[Proof of Theorem~\ref{thm:tt}] We may consider the $\beta$-transparent BTM as a \textit{randomly trapped random walk} (in the language of \cite{BenArous13}). We proceed by verifying the sufficient conditions for convergence of randomly trapped random walks given in \cite{BenArous13}.

Suppose first that $\beta \ge 1$. Conditional on $\tau_0$, we have
\[ m( \pi_0 ) :=\int_0^\infty t \pi_0(dt)= \left(\tau_0^{1-\beta}+1-\tau_0^{-\beta}\right)\mathbf{1}_{\{\tau_0\geq 1\}}+\tau_0\mathbf{1}_{\{\tau_0<1\}}\le 2 .\]
Averaging over $\tau_0$ then gives $ \mu = \mathbf{E}( m(\pi_0)) \le  2 < \infty$. Hence the conditions of \cite[Theorem 2.9]{BenArous13} are satisfied, giving the result.

Suppose now that $\beta \in (0,1)$. Proposition~\ref{prop:lp} gives precisely the assumptions of \cite[Theorem 2.11]{BenArous13} (see also \cite[Remark 2.12]{BenArous13}), which immediately yields the result. The representation for the slowly varying function $\ell$ given in Remark~\ref{rem:svfrep} is also evident from \cite[Theorem 2.11]{BenArous13} and the proof of Proposition~\ref{prop:lp}.
\end{proof}

\bigskip

\appendix

\section{}
\label{sec:appendix}

\subsection{Topologies on the space of real-valued c\`{a}dl\`{a}g functions}
The purpose of this section is to describe the Skorohod topologies $J_1$ and $M_1$, as introduced in \cite{Skorokhod56}, and the non-Skorohod topology $L_{1,\rm{loc}}$ on the Skorohod space $D(\mathbb{R}^+)$ of real-valued c\`{a}dl\`{a}g functions on $\mathbb{R}^+$, in which our main results are proved; see \cite{Billingsley99, Whitt02} for a fuller account. Figure~\ref{fig:topology} gives a graphical illustration of the different kinds of discontinuities that the three topologies allow for sequences of convergent functions.

We first define convergence in the respective topologies on the Skorohod space $D([0, T])$ of real-valued c\`{a}dl\`{a}g functions on $[0, T]$, for fixed $T$.

$\mathbf{J}_1$: A sequence of functions $f_n \in D([0, T])$ converges to a function $f \in D([0, T])$ in the $J_1$ topology if there exists a sequence $\alpha_n: [0, T] \to [0, T]$ of continuous and one-to-one maps such that
\[ \sup_{t \le T}| \alpha_n(t) - t | \to 0  \quad \text{and} \quad \sup_{t \le T} | f_n( \alpha_n(t)) - f(t) | \to 0. \]
Note that the $J_1$ topology extends the usual topology of uniform convergence over compact time intervals by allowing jumps in $f$ to be matched by jumps in $f_n$ that occur at slightly different times, as long as these differences are negligible in the limit.

$\mathbf{M}_1$: For a function $f \in D([0, T])$ define its graph $\mathcal{G}^f \subset \mathbb{R}^+ \times \mathbb{R}$ to be the ordered set consisting of the function $f$ and the line segments
\[ \bigcup_{0\leq t \leq T} \, \{ \lambda f(t^-) + (1-\lambda) f(t) : 0 \le \lambda \le 1 \}\]
connecting each point of discontinuity of $f$. We remark that $\mathcal{G}^f$ can be continuously parameterised over $t \in [0, 1]$ in the natural way; let such a parameterisation be $\mathcal{G}^f(t)$. A sequence of functions $f_n \in D([0, T])$ converges to a function $f \in D([0, T])$ in the $M_1$ topology if there exists a sequence $\alpha_n: [0, 1] \to [0, 1]$ of continuous and one-to-one maps such that
\[ \sup_{t \le 1} \max_{i = 1, 2} | \pi_i \mathcal{G}^{f_n} (\alpha_n(t)) - \pi_i \mathcal{G}^{f} (t) | \to 0   \]
where $\pi_1$ and $\pi_2$ are the projections of the graph onto the domain and codomain coordinate respectively. Note that the $M_1$ topology extends the $J_1$ topology by allowing jumps in $f$ to be matched by multiple jumps in $f_n$ of lesser magnitude as long as they are essentially monotone and occur in negligible time in the limit.

$\mathbf{L}_{1, \rm{loc}}$: A sequence of functions $f_n \in D([0, T])$ converges to a function $f \in D([0, T])$ in the $L_{1,\rm{loc}}$ topology if
\[ \int_{t \le T} | f_n(t) - f(t) |dt \to 0. \]
Note that the $L_{1,\rm{loc}}$ topology extends both the $J_1$ and the $M_1$ topologies by allowing excursions in $f_n$ that are not present in $f$, as long as they are of negligible magnitude in the $L_1$ sense in the limit.

\begin{figure}[t]
\label{fig:topology}
\begin{tikzpicture}
\begin{scope}
\draw [very thick] (0, 0) -- (1,0);
\draw [very thick, dashed] (1, 0) node[anchor=north] {$a_n$} -- (1,3);
\filldraw (1,3) circle (3pt);
\draw [very thick] (1, 3) -- (3,3);
\draw [thin, dashed] (2, 0) node[anchor=north] {$1$} -- (2,3);
\draw [thin, dashed] (0.2, 3) node[anchor=east] {$1$} -- (1,3);
\draw [->] (1.2, -0.2) -- (1.5, -0.2);
\draw [very thick] (1.5, 3.2) node[anchor=south] {$J_1$} -- (1.5, 3.2);
\end{scope}
\begin{scope}[xshift=110pt]
\draw [very thick] (0, 0) -- (1,0);
\draw [very thick, dashed] (1, 0) node[anchor=north] {$a_n$} -- (1, 1.5);
\filldraw (1, 1.5) circle (3pt);
\draw [very thick] (1, 1.5) -- (1.5, 3) -- (3, 3);
\draw [thin, dashed] (2, 0) node[anchor=north] {$1$} -- (2,3);
\draw [thin, dashed] (0.2, 1.5) node[anchor=east] {$\frac{1}{2}$} -- (1, 1.5);
\draw [thin, dashed] (0.2, 3) node[anchor=east] {$1$} -- (1.5,3);
\draw [->] (1.2, -0.2) -- (1.5, -0.2);
\draw [very thick] (1.5, 3.2) node[anchor=south] {$M_1 \,$, but not $J_1$} -- (1.5, 3.2);
\end{scope}
\begin{scope}[xshift=220pt]
\draw [very thick] (0, 0) -- (1,0) node[anchor=north] {$a_n$};
\draw [very thick, dashed] (1, 0) node[anchor=north] {$a_n$} -- (1, 1);
\filldraw (1, 1) circle (3pt);
\draw [very thick] (1, 1) -- (1.4, 3) -- (1.8, 2) -- (2.6, 2) -- (3.4, 0) ;
\draw [very thick] (3.4, 2) -- (3.7, 2);
\draw [thin, dashed] (2.6, 0) node[anchor=north] {$2$} -- (2.6, 2);
\filldraw (3.4, 2) circle (3pt);
\draw [very thick, dashed]  (3.4, 0) node[anchor=north] {$b_n$} -- (3.4, 2);
\draw [thin, dashed] (1.8, 0) node[anchor=north] {$1$} -- (1.8,2);
\draw [thin, dashed] (0.2, 3) node[anchor=east] {$\frac{3}{2}$} -- (1.4,3);
\draw [thin, dashed] (0.2, 1) node[anchor=east] {$\frac{1}{2}$} -- (1,1);
\draw [thin, dashed] (0.2, 2) node[anchor=east] {$1$} -- (1.8, 2);
\draw [->] (1.2, -0.2) -- (1.5, -0.2);
\draw [->] (3.2, -0.2) -- (2.9, -0.2);
\draw [very thick] (1.8, 3.2) node[anchor=south] {$L_{1,\rm{loc}}$, but not $J_1$ or $M_1$} -- (1.8, 3.2);
\end{scope}
\end{tikzpicture}
\caption{Examples of sequences of functions in $D(\mathbb{R}^+)$ that converge to the function $\mathbf{1}_{[1, \infty)}(\cdot)$ in the $J_1$, $M_1$ and $L_{1,\rm{loc}}$ topologies respectively, where $a_n := 1 - n^{-1}$ and $b_n := 2 + n^{-1}$.}
\end{figure}
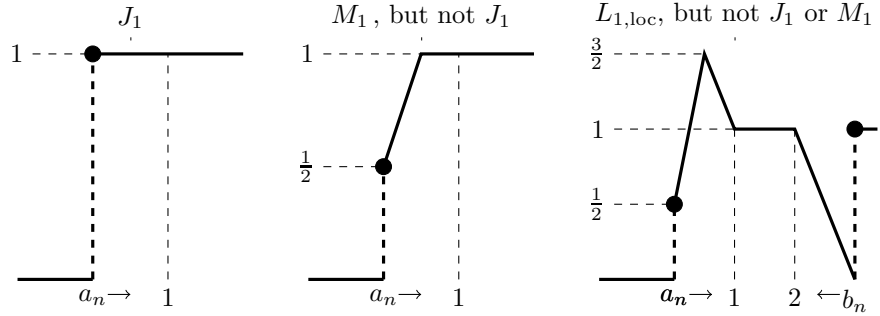

To extend the above definitions to the Skorohod space $D(\mathbb{R}^+)$, we say that a sequence of functions $f_n \in D(\mathbb{R}^+)$ converges to a function $f \in D(\mathbb{R}^+)$ in the $J_1$ (respectively $M_1$ and $L_{1,\rm{loc}}$) topology if and only if their restrictions to $[0, T]$ converge with respect to the $J_1$ (respectively $M_1$ and $L_{1,\rm{loc}}$) topology on $D([0, T])$ for every continuity point $T$ of $f$.

To summarise, we have that the $J_1$, $M_1$ and $L_{1,\rm{loc}}$ topologies are strictly ordered in the following sense, where we write $\stackrel{J_1}{\to}$, $\stackrel{M_1}{\to}$ and $\stackrel{L_1}{\to}$ for convergence in the relevant topologies.

\begin{proposition}
For a sequence of functions $f_n\in D(\mathbb{R}^+)$ and a function $f \in D(\mathbb{R}^+)$,
\[ \ f_n \stackrel{J_1}{\to} f \ \implies \ f_n \stackrel{M_1}{\to}  f \  \implies \ f_n \stackrel{L_1}{\to}  f,\]
but none of the converse implications hold in general.
\end{proposition}

\subsection{Convergence lemmas}

The purpose of this section is to present the characterisations of convergence in the topologies of interest on $D(\mathbb{R}^+)$ that we will appeal to in the proofs of our main results. Let $w_t^n, x_t^n$ and $y_t^n$ be sequences of stochastic processes in $D(\mathbb{R}^+)$. Throughout we will assume that each $w_t^n$, $x_t^n$ and $y_t^n$ is non-decreasing with
\[ \lim_{n \to \infty} w_0^n =  \lim_{n \to \infty} x_0^n =  \lim_{n \to \infty} y_0^n = 0 \]
in probability.

We will also suppose that there is a limiting stochastic process $x_t \in D(\mathbb{R}^+)$ such that, as $n \to \infty$,
\[ (w_t^n)_{t \ge 0} \stackrel{J_1}{\Rightarrow} (x_t)_{t \ge 0}  \quad \text{and} \quad  (y_t^n)_{t \ge 0} \stackrel{J_1}{\Rightarrow} (x_t)_{t \ge 0},\]
where we recall that $\stackrel{J_1}{\Rightarrow}$, $\stackrel{M_1}{\Rightarrow}$ and $\stackrel{L_1}{\Rightarrow}$ denotes weak convergence in the $J_1$ and $M_1$ and $L_{1,\rm{loc}}$ topologies respectively.

We first give sufficient conditions under which the stochastic processes $x_t^n$ also converge weakly to the limit process $x_t$ in the $J_1$ and $M_1$ topologies respectively. For technical reasons we state our results in a way that allows for an auxiliary process to converge simultaneously. For a sequence of probability measures on $D(\mathbb{R}^+) \times D(\mathbb{R}^+)$, denote by
\[\quad \stackrel{J_1/J_1}{\Rightarrow} \ , \quad \stackrel{M_1/J_1}{\Rightarrow} \quad \text{and} \quad \stackrel{L_1/J_1}{\Rightarrow}\]
weak convergence of the first component in the $J_1$, $M_1$ and $L_{1,\rm{loc}}$ topologies respectively, and the simultaneous weak convergence of the second component in the $J_1$ topology. Then let $z^n_t$ be an auxiliary sequence of stochastic processes in $D(\mathbb{R}^+)$ such that, as $n \to \infty$,
\begin{equation}\label{assumpconv}
  (w_t^n, \, z^n_t)_{t \ge 0} \stackrel{J_1 / J_1}{\Rightarrow} (x_t, \, z_t)_{t \ge 0}  \quad \text{and} \quad  (y_t^n, \, z^n_t)_{t \ge 0} \stackrel{J_1 / J_1}{\Rightarrow} (x_t, \, z_t)_{t \ge 0}, 
  \end{equation}
for a limit process $z_t$ in $D(\mathbb{R}^+)$.

\begin{lemma}[Uniform convergence in space implies $J_1$ convergence]
\label{lem:sconv1}
Assume that, for any $T > \delta > 0$, as $n \to \infty$,
\begin{align}
\label{eq:conv1}
\sup_{t \in [\delta, T]} | x^n_t - w^n_t | \to 0 \qquad \text{in probability}.
\end{align}
Then, as $n \to \infty$,
\[ (x^n_t, \, z^n_t)_{t \ge 0} \stackrel{J_1/J_1}{\Rightarrow} (x_t, \, z_t)_{t \ge 0}.\]
\end{lemma}
\begin{proof} The convergence of the finite dimensional distributions (where the time indices are points at which $x$ and $z$ are continuous almost-surely) immediately follows from the convergence at (\ref{assumpconv}) and \eqref{eq:conv1}; we need only show the tightness of the sequence $x^n_t$ in the $J_1$ topology. In particular, we need to show that, for each $T>0$,
\[\lim_{\lambda\rightarrow \infty}\limsup_{n\rightarrow \infty}\mathbf{P}\left(\sup_{t\leq T}\left|x_t^n\right|\geq \lambda\right)=0\]
and, for each $\varepsilon>0$,
\[\lim_{\delta\rightarrow0}\limsup_{n\rightarrow \infty}\mathbf{P}\left(\inf_{\substack{0=t_0<t_1<\dots<t_m=T:\\\min_i(t_i-t_{i-1})>\delta}}\sup_{i=1,\dots, m}\sup_{s,t\in[t_{i-1},t_i)}\left|x_s^n-x_t^n\right|\geq \varepsilon\right)=0\]
(see \cite[Theorem 16.8]{Billingsley99}). Now, the convergence of $w_t^n$ in the $J_1$ topology implies that the two conditions in \cite[Theorem 16.8]{Billingsley99} are satisfied for $w_t^n$. It is an elementary exercise to check from this, the uniform convergence in equation \eqref{eq:conv1}, and the assumption that each $w_t^n$ and $x_t^n$ is non-decreasing with $w_0^n$ and $x_0^n$ converging to zero, that the conditions are also satisfied for $x^n_t$.
\end{proof}

\begin{lemma}[Squeeze convergence in space and time implies $M_1$ convergence]
\label{lem:sconv2}
Assume that, for any $T > \delta > 0$, $t \ge 0$ and $n \in \mathbb{N}$ there exists a random, measurable $s^n_t \in [t, t + \delta]$ such that, as $n \to \infty$,
\begin{align}
\label{eq:conv2}
 \PP \left( w^n_{s^n_t} - \delta  < x^n_{s^n_t} < y^n_{s^n_t} + \delta  \quad \text{for all } t \in [\delta, T] \, \right) \to 1.
 \end{align}
Further, assume the limit process $(x_t,z_t)$ is almost-surely continuous at each fixed time $t$. Then, as $n \to \infty$,
\[ (x^n_t, \, z^n_t)_{t \ge 0} \stackrel{M_1/J_1}{\Rightarrow} (x_t, \, z_t)_{t \ge 0}.\]
\end{lemma}
\begin{proof} The convergence of the finite dimensional distributions immediately follows from (\ref{assumpconv}), the squeeze convergence in equation \eqref{eq:conv2}, and the fact that $w_t^n$, $x_t^n$ and $y_t^n$ are all non-decreasing; we need only show the tightness of the sequence $x^n_t$ in the $M_1$ topology. Using the characterisation of tightness in \cite[Theorem 12.12.3]{Whitt02}, we need to show in particular that, for each $T > 0$,
\[\lim_{\lambda\rightarrow \infty}\limsup_{n\rightarrow \infty}\mathbf{P}\left(\sup_{t\leq T}\left|x_t^n\right|\geq \lambda\right)=0\]
and, for each $\varepsilon>0$,
\[\lim_{\delta\rightarrow0} \limsup_{n\rightarrow \infty}\mathbf{P}\left(\sup_{t \in [0, T]} \sup_{ \substack{ \max\{0, t- \delta\} \le t_1 \\ < t_2 < t_3 \le \min\{t + \delta, T\} } } \left\{  \left| \left| x_{t_2}^n- \left[ x_{t_1}^n, x_{t_3}^n \right] \right| \right| \right\}  > \varepsilon \right)  \]
where $\left| \left| x_{t_2}^n- \left[ x_{t_1}^n, x_{t_3}^n \right] \right| \right|$ denotes the Hausdorff distance in  $\mathbb{R}^+ \times \mathbb{R}$ between the point $(t_2, x_{t_2}^n)$ and the line segment joining the points $(t_1, x_{t_1}^n)$ and $(t_3, x_{t_3}^n)$. Now, the convergence of $w_t^n$ and $x_t^n$ in the $J_1$ (and hence $M_1$) topology imply that the above two conditions are satisfied for $w_t^n$ and $y_t^n$. A combination of this, equation \eqref{eq:conv2} and the fact that each $w_t^n$, $x_t^n$ and $y_t^n$ is non-decreasing with $w_0^n$, $x_0^n$ and $y_0^n$ converging to zero, then implies that the two conditions are also satisfied for $x_t^n$.
\end{proof}

Finally, we give basic inversion and composition lemmas for the topologies. We henceforth assume that $x_t^n$ and $z_t^n$ are deterministic functions in $D(\mathbb{R}^+)$ such that $x_t^n$ is non-decreasing with
\[ \lim_{n \to \infty} x_0^n = 0.  \]
We further assume that, as $n \to \infty$,
\begin{align}
\label{eq:prod}
(x_t^n, \, z^n_t)_{t \ge 0} \stackrel{J_1/J_1}{\to} (x_t, \, z_t)_{t \ge 0}
\end{align}
for some $x_t,z_t\in D(\mathbb{R}^+)$, where the notation means $\stackrel{J_1/J_1}{\to}$ that the first component converges in the $J_1$ topology and the second in the $J_1$ topology. (We will use similar notation when $J_1$ is replaced by $M_1$ or $L_{1,\rm{loc}}$.) Furthermore, we assume that the limit process $z_t \in C(\mathbb{R}^+)$, the space of real-valued continuous functions on $\mathbb{R}^+$ (so that the convergence of the second coordinate can actually be considered as uniform convergence over compact time intervals). The reason that we insist on the presence of an auxiliary process is to take advantage of the composition lemma that we state below.

\begin{lemma}[Inversion lemma]
\label{lem:inv}
Let $v_t^n$ and $v_t$ denote the right-continuous inverses of $x_t^n$ and $x_t$ respectively, and further assume that $v_t$ is based at the origin. Then, as $n \to \infty$,
\[ (v^n_t, \, z^n_t)_{t \ge 0} \stackrel{M_1/J_1}{\to} (v_t, \, z_t)_{t \ge 0}.\]
The same conclusion holds if we weaken the assumption in equation \eqref{eq:prod} to
\[ (x_t^n, \, z^n_t)_{t \ge 0} \stackrel{M_1/J_1}{\to} (x_t, \, z_t)_{t \ge 0}.\]
\end{lemma}
\begin{proof}
This is a consequence of the continuity of right-continuous inverses in the $M_1$ topology (see \cite[Corollary 13.6.5]{Whitt02}; the continuity at zero follows from the assumptions that $x_0^n$ converges to zero and that $v_t$ is based at the origin).
\end{proof}

\begin{remark}
Note that, even under the assumption of equation \eqref{eq:prod}, the conclusion of Lemma~\ref{lem:inv} does not hold in general in the stronger $J_1$ topology; see the discussion in \cite[Example 13.6.1]{Whitt02}.
\end{remark}

\begin{lemma}[Composition lemma]
\label{lem:comp}
As $n \to \infty$,
\[ (z^n_{x^n_t})_{t \ge 0} \stackrel{J_1}{\to} (z_{x_t})_{t \ge 0}.\]
If instead the assumption in equation \eqref{eq:prod} is weakened to
\[ (x_t^n, \, z^n_t)_{t \ge 0} \stackrel{M_1/J_1}{\to} (x_t, \, z_t)_{t \ge 0},\]
then we may only conclude that, as $n \to \infty$,
\[ (z^n_{x^n_t})_{t \ge 0} \stackrel{L_1}{\to} (z_{x_t})_{t \ge 0}.\]

\end{lemma}
\begin{proof}
The first statement is standard, proven for example in \cite[Theorem 3.1]{Whitt80} (see also \cite[Theorem 13.2.2]{Whitt02}). As for the second statement, we start by noting that
\begin{equation}\label{twoterms}
\int_{t\leq T}\left|z_{x_t^n}^n-z_{x_t}\right|dt\leq \int_{t\leq T}\left|z_{x_t^n}^n-z_{x_t^n}\right|dt+\int_{t\leq T}\left|z_{x_t^n}-z_{x_t}\right|dt.
\end{equation}
Now, since $x_t^n\to x_t$ in the $M_1$ topology, we must have that $\sup_{n}\sup_{t\leq T}x_t^n$ is bounded above by some $T_1<\infty$. Hence the first term on the right-hand side of (\ref{twoterms}) satisfies
\[\int_{t\leq T}\left|z_{x_t^n}^n-z_{x_t^n}\right|dt\leq T\sup_{t\leq T_1}\left|z_{t}^n-z_{t}\right|,\]
which converges to zero as $n\to\infty$ by the uniform convergence of $z_t^n$ to $z_t$ over compact time intervals. We now deal with the second term on the right-hand side of (\ref{twoterms}). First note that we can also assume that $\sup_{t\leq T}x_t\leq T_1<\infty$ (adjusting $T_1$ if necessary). Moreover, the continuity of $z$ yields that $\sup_{t\leq T_1}|z_t|\leq C<\infty$. Putting these bounds together, we find that, for every $\varepsilon>0$,
\[\int_{t\leq T}\left|z_{x_t^n}-z_{x_t}\right|dt\leq T\sup_{\substack{s,t\leq T_1:\\|s-t|<\varepsilon}}\left|z_{s}-z_{t}\right|+
C\int_{t\leq T} \mathbf{1}_{\{|x_t^n-x_t|\geq \varepsilon\}}dt.\]
Since $x_t^n$ converges to $x_t$ in the $M_1$ topology, the same is true in the $L_{1,\rm{loc}}$ topology, which implies that the second term here converges to 0 as $n\to\infty$. Again appealing to the continuity of $z$, the first term can be made arbitrarily small by taking $\varepsilon$ small. This confirms that
\[\lim_{n\to\infty}\int_{t\leq T}\left|z_{x_t^n}^n-z_{x_t}\right|dt=0,\]
as desired.
\end{proof}

\begin{remark}
The assumption that $z_t \in C(\mathbb{R}^+)$ is essential for the first conclusion of the previous result. Indeed, it no longer holds in general if we assume only that $z_t \in D(\mathbb{R}^+)$ with convergence in the $J_1$ topology; see the discussion in \cite[Example 13.2.2]{Whitt02}.
\end{remark}

\begin{remark}
The fact that the second convergence statement in Lemma~\ref{lem:comp} fails to hold in any of the Skorohod topologies lies at the heart of why we resort to the coarser non-Skorohod $L_{1,\rm{loc}}$ topology in Theorems \ref{thm:fltX} and \ref{thm:assumptX}; see the discussion in \cite[Example 13.2.4]{Whitt02}.
\end{remark}

\bibliography{paper}{}

\begin{thebibliography}{10}

\bibitem{BenArous13}
G.~{Ben Arous}, M.~Cabezas, J.~\v{C}ern\'{y}, and R.~Royfman.
\newblock Randomly trapped random walks.
\newblock {\em Ann. Probab. (to appear)}, 2014.

\bibitem{BenArous12}
G.~{Ben Arous} and O.~G{\"u}n.
\newblock Universality and extremal aging for dynamics of spin glasses on
  subexponential time scales.
\newblock {\em Comm. Pure Appl. Math.}, 65:77--127, 2012.

\bibitem{BenArous06}
G.~{Ben Arous} and J.~\v{C}ern\'{y}.
\newblock Dynamics of trap models.
\newblock {\em Math. Stat. Physics Lecture Notes -- Les Houches Summer School},
  83, 2006.

\bibitem{Billingsley99}
P.~Billingsley.
\newblock {\em Convergence of Probability Measures}.
\newblock John Wiley \& Sons, 1999.

\bibitem{Bingham87}
N.H. Bingham, C.M. Goldie, and J.L. Teugels.
\newblock {\em Regular Variation}.
\newblock Cambridge University Press, 1987.

\bibitem{Bouchaud92}
J.P. Bouchaud.
\newblock Weak ergodicity breaking and aging in disordered systems.
\newblock {\em J. Phys. I (France)}, 2:1705--1713, 1992.

\bibitem{Bovier13}
A.~Bovier, V.~Gayrard, and A.~\v{S}vejda.
\newblock Convergence to extremal processes in random environments and extremal
  ageing in {SK} models.
\newblock {\em Probab. Theory Relat. Fields}, 157:251--283, 2013.

\bibitem{Croydon13}
D.~Croydon, A.~Fribergh, and T.~Kumagai.
\newblock Biased random walk on critical {G}alton--{W}atson trees conditioned
  to survive.
\newblock {\em Probab. Theory Relat. Fields}, 157:453--507, 2013.

\bibitem{FIN02}
L.R.G. Fontes, M.~Isopi, and C.M. Newman.
\newblock Random walks with strongly inhomogeneous rates and singular
  diffusions: Convergence, localization and aging in one dimension.
\newblock {\em Ann. Probab.}, 30(2):579--604, 2002.

\bibitem{Fontes13}
L.R.G. Fontes and P.~Mathieu.
\newblock On the dynamics of trap models in $\mathbb{Z}^d$.
\newblock {\em Proc. London Math. Soc.}, 108(6):1562--1592, 2013.

\bibitem{Gun13}
O.~G{\"u}n.
\newblock Extremal aging for trap models.
\newblock {\em arXiv:1312.1137}, 2013.

\bibitem{Kall}
O.~Kallenberg.
\newblock {\em Foundations of modern probability}.
\newblock Springer-Verlag, New York, 2002.

\bibitem{Kasahara86}
Y.~Kasahara.
\newblock A limit theorem for sums of i.i.d. random variables with slowly
  varying tail probability.
\newblock {\em J. Math. Kyoto. Univ.}, 37:197--205, 1986.

\bibitem{Lamperti64}
J.~Lamperti.
\newblock On extreme order statistics.
\newblock {\em Ann. Math. Statist.}, 35:1726--1737, 1964.

\bibitem{LL01}
E.~H. Lieb and M.~Loss.
\newblock {\em Analysis}, volume~14 of {\em Graduate Studies in Mathematics}.
\newblock American Mathematical Society, 2001.

\bibitem{Mathieu14}
P.~Mathieu and J.-C. Mourrat.
\newblock Aging of asymmetric dynamics on the random energy model.
\newblock {\em Probab. Theory Relat. Fields}, 2014.

\bibitem{Muirhead14}
S.~Muirhead.
\newblock Two-site localisation in the {B}ouchaud trap model with
  slowly-varying traps.
\newblock {\em arXiv:1402.4983}, 2014.

\bibitem{Resnick87}
S.~Resnick.
\newblock {\em Extreme Values, Regular Variation, and Point Processes}.
\newblock Springer, 1987.

\bibitem{Skorokhod56}
A.V. Skorokhod.
\newblock Limit theorems for stochastic processes.
\newblock {\em Th. Probab. Appl.}, 1:261--290, 1956.

\bibitem{Whitt80}
W.~Whitt.
\newblock Some useful functions for functional limit theorems.
\newblock {\em Math. Oper. Res.}, 5(1):67--85, 1980.

\bibitem{Whitt02}
W.~Whitt.
\newblock {\em Stochastic-Process Limits}.
\newblock Springer, 2002.

\end{thebibliography}
\bibliographystyle{plain}

\end{document}